\documentclass[a4paper,11pt,reqno]{amsart}

\usepackage[ansinew]{inputenc}
\usepackage[english]{babel}
\usepackage{amsmath}
\usepackage{amssymb}
\usepackage{amsthm}
\usepackage{mathrsfs}
\usepackage[all]{xy}
\usepackage{verbatim}
\usepackage{arydshln}

\usepackage{mathtools}
\usepackage{ifthen}
\usepackage{tikz}
\usepackage{todonotes}
\usepackage[final,pdfauthor={Matthew Gelvin, Sune Precht Reeh},pdftitle={}]{hyperref}
\usepackage{bookmark}

\numberwithin{equation}{section}
\numberwithin{figure}{section}
\numberwithin{table}{section}
\allowdisplaybreaks

\theoremstyle{plain}
\newtheorem{theorem}{Theorem}[section]
\newtheorem{prop}[theorem]{Proposition}
\newtheorem{lemma}[theorem]{Lemma}
\newtheorem{cor}[theorem]{Corollary}
\newtheorem{conj}[theorem]{Conjecture}
\newtheorem{fact}[theorem]{Fact}

\theoremstyle{definition}
\newtheorem{definition}[theorem]{Definition}

\newtheorem{notation}[theorem]{Notation}

\theoremstyle{remark}
\newtheorem{remark}[theorem]{Remark}

\newtheorem{example}[theorem]{Example}
\newtheorem{counter}[theorem]{Counterexample}

\DeclareMathOperator{\Aut}{Aut}

\DeclareMathOperator{\Inn}{Inn}
\DeclareMathOperator{\Out}{Out}

\DeclareMathOperator{\Stab}{Stab}

\DeclareMathOperator{\Span}{Span}

\DeclareMathOperator{\Hom}{Hom}

\DeclareMathOperator{\Syl}{Syl}

\newcommand{\id}{\mathrm{id}}

\newcommand{\Mor}{\mathrm{Mor}}
\newcommand{\iso}{\mathrm{iso}}

\newcommand{\ph}{\varphi}

\newcommand{\x}{\times}

\renewcommand{\tilde}{\widetilde}

\newcommand{\Mid}{\bigm|}

\newcommand{\lc}[1]{\prescript{#1\!}{}{}}

\newcommand{\xra}{\xrightarrow}

\def\cF{\mathcal F}
\def\cL{\mathcal L}
\def\cN{\mathcal N}

\def\sI{\mathscr I}

\def\fg{\mathfrak g}\def\fh{\mathfrak h}\def\fk{\mathfrak k}
\def\fI{\mathfrak I}

\def\proj{\mathfrak{pr}}

\def\ZZ{\mathbb Z}

\def\id{\mathrm{id}}

\DeclarePairedDelimiter{\abs}\lvert\rvert

\renewcommand{\theenumi}{(\roman{enumi})}\renewcommand{\labelenumi}{\theenumi}

\usetikzlibrary{arrows,matrix,decorations,positioning,calc}
\tikzset{dot/.style={circle,fill=black,thick,inner sep=0pt,minimum size=1mm,draw}}
\tikzset{arrow/.style={semithick,>=stealth',shorten >=1pt,shorten <=1pt}}
\tikzset{equal/.style={arrow,double distance=2pt}}

\title{Minimal characteristic bisets for fusion systems}
\author[M. Gelvin]{Matthew Gelvin}
\email{mgelvin@gmail.com}
\author[S. P. Reeh]{Sune Precht Reeh}
\address{Department of Mathematical Sciences,  University of Copenhagen, Copenhagen, Denmark}
\email{spr@math.ku.dk}
\subjclass[2010]{}

\thanks{Supported by the Danish National Research Foundation through the Centre for Symmetry and Deformation (DNRF92).}

\begin{document}

\begin{abstract}
  {} We show that every saturated fusion system $\cF$ has a unique minimal $\cF$-characteristic biset $\Lambda_\cF$.  We examine the relationship of $\Lambda_\cF$ with other concepts in $p$-local finite group theory:  In the case of a constrained fusion system, the model for the fusion system is the minimal $\cF$-characteristic biset, and more generally, any centric linking system  can be identified with the $\cF$-centric part of $\Lambda_\cF$ as bisets.  We explore the grouplike properties of $\Lambda_\cF$, and  conjecture an identification of normalizer subsystems of $\cF$ with subbisets of $\Lambda_\cF$.
\end{abstract}

\maketitle

\section{Introduction}

{
\renewcommand\thetheorem{\Alph{theorem}}

If $S$ is a Sylow $p$-subgroup of a finite group $G$, we talk about the fusion system $\cF_S(G)$ as an organizational framework for understanding the $p$-local structure of $G$.  The fusion data is encoded as a category:  The objects of $\cF_S(G)$ are the subgroups of $S$, and the morphisms are the maps between subgroups induced by conjugation in $G$.
More generally, Puig introduced the notion of an abstract fusion system on $S$:  This is again a category $\cF$ with objects the subgroups of $S$ and morphism certain injective group maps between subgroups (see Section \ref{secFusionSystems}).

An abstract fusion system does not necessarily arise from a group in this manner, but we still think of the morphisms in $\cF$ as given by the conjugation action of some grouplike object on the subgroups of $S$. The notion of a characteristic biset turns this perspective around, and considers how $S$ acts on the object that does the conjugating.

For $S\in\Syl_p(G)$ and the fusion system $\cF_S(G)$ realized by $G$'s conjugation action on $S$, we can  ask how $S$ acts on $G$ by left and right multiplication.  That is, we consider the $(S,S)$-biset $_SG_S$.  For $g\in G$, if $(b,a)\in S\times S$ is such that  $b\cdot g=g\cdot a$, then $b={}^g a$.  In other words, fusion data ($b={}^ga$) is encoded in the biset structure ($b\cdot g=g\cdot a$).  This justifies calling $_SG_S$ a \emph{characteristic biset} for  $\cF_S(G)$.

Linckelmann and Webb extracted the features of $_SG_S$ that are essential for understanding the fusion system $\cF_S(G)$, resulting in a notion of characteristic bisets for any abstract fusion system $\cF$.
Fix a $p$-group $S$, a fusion system $\cF$ on $S$, and an $(S,S)$-biset $\Omega$. $\Omega$ is then a characteristic biset for $\cF$ if:

\begin{itemize}
\item[(0)] $\Omega$ is free both as a left and right $S$-set. \vspace{5pt}
\begin{itemize}
\item[]This implies that any $\omega\in\Omega$ has  stabilizer $\{(b,a)\in S\times S\mid b\cdot\omega=\omega\cdot a\}$ of the form $(P,\varphi):=\{(\varphi(a),a)\mid a\in P\}$ for $P$ a subgroup of $S$ and $\varphi\colon P\hookrightarrow S$ some group injection.

Heuristically, this says that $\omega$ ``conjugates'' $a$ to $\varphi(a)$.
\end{itemize}\vspace{5pt}
\item[(1)]  If $\omega\in\Omega$ has stabilizer $(P,\varphi)$, then $\varphi$ is a morphism of $\cF$.\vspace{5pt}
\begin{itemize}
\item[]This means that all the conjugation induced by $\Omega$ is in $\cF$.
\end{itemize}\vspace{5pt}
\item[(2)]  For subgroups $P$  of $S$, $\cF$-morphisms $\varphi\colon P\to S$, and $\cF$-isomorphisms $\eta_1\colon Q\xrightarrow{\cong} P$, $\eta_2\colon \varphi P\xrightarrow{\cong} R$, there is an equality of fixed-point set orders:  $\abs*{\Omega^{(P,\varphi)}}=\abs*{\Omega^{(Q,\eta_2\varphi\eta_1)}}$.\vspace{5pt}
\begin{itemize}
\item[]This condition generalizes the fact that, if $G$ acts on a set $X$, then conjugate subgroups of $G$ have fixed-point sets of equal size.
\end{itemize}\vspace{5pt}
\item[(3)]  $\abs\Omega/\abs S$ is prime to $p$.\vspace{5pt}
\begin{itemize}
\item[]This Sylow condition generalizes $S\in\Syl_p(G)$.
\end{itemize}
\end{itemize}

The connection between a fusion system $\cF$ and an associated $\cF$-characteristic biset is very strong:
\begin{itemize}
\item If $\cF$ is \emph{saturated} (i.e., it satisfies the axioms needed to make $\cF$ look like the fusion induced by a finite group), then there exists a characteristic biset (\cite{BLO2}).
\item If a characteristic biset for $\cF$ exists, then $\cF$ is saturated (\cite{PuigBook}, also see \cite{KSFusionBurnside} for a $p$-localized version).
\item As suggested by Axioms (1) and (2), the characteristic biset  determines $\cF$.
\end{itemize}
If we allowed ourselves to think about virtual bisets in the double Burnside ring of $S$, and $p$-localized, then the converse to the last point would be true:  Every saturated fusion system determines and is determined by a unique \emph{characteristic idempotent} in $A(S,S)_{(p)}$, see \cite{RagnarssonClassifyingSpectra}. However, motivated by Park's Theorem that an $\cF$-characteristic biset  gives rise to ambient, finite (but not necessarily Sylow) supergroup realizing $\cF$ (\cite{ParkSingleGroupFusion}) and subsequent work investigating smallest $\cF$-characteristic biset orders in certain examples (\cite{ParkMinimalCharacteristicBisets}), we will opt to instead remain in the world of honest bisets.

For us then, the uniqueness of $\cF$-characteristic bisets always fails:   If $G$ and $H$ both contain $S$ as a Sylow $p$-subgroup, and if $\cF_S(G)=\cF_S(H)=:\cF$, then both $_SG_S$ and $_SH_S$ are characteristic bisets for $\cF$, but they need not be equal (e.g., $H=G\times K$ for $K$ your favorite $p'$-group).    While a characteristic biset determines the fusion system, the fusion system does not determine the characteristic biset.

This paper proposes to solve this indeterminacy problem:
In Theorem \ref{parameterization of semicharacteristic bisets} we give a complete parameterization of all $\cF$-characteristic bisets, which in particular implies

\begin{theorem}[Corollary \ref{minimal characteristic biset}]\label{thmIntro:minimal characteristic biset}
Every saturated fusion system $\cF$ has a unique minimal characteristic biset $\Lambda_\cF$.
\end{theorem}

Here, minimality means that if if $\Omega$ is any  $\cF$-characteristic biset, then $\Lambda_\cF\subseteq\Omega$ as $(S,S)$-bisets.  This makes $\Lambda_\cF$  the most natural choice of $\cF$-characteristic biset, and we argue that it should be thought of as \emph{the} characteristic biset by proving several additional Theorems \ref{thmIntro:Models}-\ref{thm:IntroKNormalizers} justifying this choice.

The preliminary Sections \ref{secFusionSystems} and \ref{secBisets} contain the necessary background material for this paper. Corollary \ref{normalizer size in SxS} in particular will play an essential role in identifying  $\Lambda_\cF$.

Section \ref{sec:monoid of F-sets} contains the main technical background relating $S$-sets to $\cF$-fusion needed in our search for $\Lambda_\cF$. If $\cF$ is a saturated fusion system on $S$, and $X$ is a finite $S$-set, we say that $X$ is \emph{$\cF$-stable} if for all $\cF$-morphisms $\varphi\colon P\to S$, we have an equality of fixed-point set orders $\abs*{X^P}=\abs*{X^{\varphi P}}$ (cf. Axiom (2)  for characteristic bisets). Just as the transitive $G$-sets form the basis for commutative monoid of all finite $G$-sets, the first author conjectured that the commutative monoid of $\cF$-stable $S$-sets is free with basis naturally corresponding to the $\cF$-conjugacy classes of subgroups of $S$. The second author proved this in \cite{ReehStableSets}. We recall the defining features of these elements in Theorem \ref{F-basis}, and provide a new proof that they actually form  a basis in Corollary \ref{F-basis cor}.

In Section \ref{sec:parameterization of semicharacteristic bisets} we prove Theorem \ref{thmIntro:minimal characteristic biset} by rephrasing the problem as looking for a particular kind of $\cF\times\cF$-stable $S\times S$-set.

It should be emphasized that the parameterization of $\cF$-characteristic bisets, and hence the construction of $\Lambda_\cF$, relies solely on a straightforward  counting argument, inductively indexed on the objects of $\cF$.  $\Lambda_\cF$ is therefore much easier to get a hold of than most of the objects that appear in $p$-local finite group theory.  Even so, it turns out that there are deep connections between the minimal $\cF$-characteristic biset and other, more complicated structures.  We take this as further evidence for the special role played by $\Lambda_\cF$, and devote the rest of the paper to exploring these connections.

Section \ref{sec:ModelsOfConstrainedFusionSystems} examines the minimal characteristic bisets of constrained fusion systems, which we view as the building blocks from which all fusion systems are glued together.  If $O_p(\cF)$ denotes the maximal normal subgroup of $\cF$ (so that every morphism of $\cF$ extends to induce an automorphism of $O_p(\cF)$), we say that $\cF$ is \emph{constrained} if $C_S(O_p(\cF))\leq O_p(\cF)$. Constrained fusion systems always come from finite groups, and in fact among all finite groups inducing such an $\cF$ there is a well defined minimal example.  This finite group $M_\cF$ is the \emph{model} of $\cF$, which is characterized by requiring that $C_{M_\cF}(O_p(M_\cF))\leq O_p(M_\cF)$.  As the constrained fusion system $\cF$ has both a minimal characteristic biset and a minimal group inducing $\cF$, we might ask about the relationship between the two.

\begin{theorem}[Theorem \ref{models are minimal characteristic bisets}]\label{thmIntro:Models}
If $\cF$ is a constrained fusion system with minimal characteristic biset $\Lambda_\cF$ and model $M_\cF$, then $_S(M_\cF)_S=\Lambda_\cF$ as $(S,S)$-bisets.
\end{theorem}

In Section \ref{sec:LinkingSystemsAsCharacteristicBisets} we turn to more general fusion systems.  If $\cF$ is not constrained, then there is no particularly good notion of a ``minimal'' group inducing $\cF$; indeed, in the case of exotic fusion systems there may be no finite Sylow supergroup at all.  Even in these cases we can still talk about an associated $p$-local finite group, which is formed by augmenting the fusion system with an auxiliary category $\cL$, the \emph{centric linking system}.  The morphisms of $\cL$ represent group elements whose conjugation actions induce the morphisms of $\cF$; this is made precise in Chermak's notion of a \emph{partial group} (of which $\cL$ is the motivating example), which is effectively a different method of packaging the data of a linking system.

In \cite{ChermakFusionSystemsAndLocalities} it was shown that every saturated fusion system has a unique associated centric linking system, using the Classification Theorem of Finite Simple Groups.  Independent of this result, if we assume that a linking system $\cL$ exists, the axioms governing its structure allow us to define an $(S,S)$-biset structure on the set $\fI$ of nonextendable isomorphisms of $\cL$.  While $\fI$ is not $\Lambda_\cF$, we do have  $\fI\subseteq\Lambda_\cF$ as $(S,S)$-bisets.  Moreover, we can identify $\fI$ as the elements of $\cL$ that conjugate an object of $\cL$ (an \emph{$\cF$-centric subgroup}) into $S$:

\begin{theorem}[Theorem \ref{thm:LinkingSystemsAreCentricPartOfMinimalCharacteristicBiset}]\label{thm:IntroLinkingSystemAsCharacteristicBisets}
If $\cL$ is a centric linking system associated to $\cF$, then the $(S,S)$-biset of nonextendable isomorphisms $\fI$ is the $\cF$-centric part of $\Lambda_\cF$.
\end{theorem}

It should be noted that this biset $\fI$ is just the elements of the partial group $\cL$.

We interpret this result as saying that $\Lambda_\cF$ contains both more and less data than the linking system $\cL$:  Less in that only the left and right multiplications by $S$ are defined (so that $\Lambda_\cF$ does not even have a partial group structure), but more in that the minimal $\cF$-characteristic biset sees \emph{all} the subgroups of $S$ and not just the $\cF$-centric ones.  This suggests the possibility of using minimal characteristic bisets to avoid some of the nonfunctoriality of linking systems in future work.

Theorem \ref{thm:IntroLinkingSystemAsCharacteristicBisets} is a uniqueness statement about centic linking systems associated to $\cF$.  In Section \ref{sec:LinkingSystemFreeMinimalBisets} we establish a corresponding existence statement:  Without reference to a centric linking system for $\cF$, we set out to identify the $\cF$-centric part of $\Lambda_\cF$.

It turns out that the answer has a pleasingly simple form.  If $\varphi_i\colon P_i\xrightarrow{\cong} Q_i$, $i=1,2$, are $\cF$-isomorphisms, we say that $\varphi_1$ is equivalent to $\varphi_2$ if there exist $a,b\in S$ such that $\varphi_2=c_b\circ\varphi_2\circ c_a$.

\begin{theorem}[Theorem \ref{minimal centric biset structure}]\label{thm:IntroMinimalCentricBisetStructure}
The $\cF$-centric part of $\Lambda_\cF$ has one $(S,S)$-orbit for each equivalence class of nonexentable isomorphisms of $\cF$.  The orbit corresponding to the class of $\varphi\colon P\to Q$ is $S\times_{(P,\varphi)}S$.
\end{theorem}

In \cite{GelvinReehYalcinBasisOfBurnsideRing}, a general framework for computing the orbits of $\Lambda_\cF$ is developed as a special case of a much more general combinatorial argument.  The advantage of the current Theorem \ref{thm:IntroMinimalCentricBisetStructure} lies in the relative simplicity of the solution, along with the comparatively straightforward method used in the proof.

In Section \ref{sec:KNormalizers}, we close by considering the local group-theoretic properties of $\Lambda_\cF$.  Returning to the connection between point-stabilizers and conjugation from Axiom (0) of $\cF$-characteristic bisets, we define notions of centralizer and normalizer subbisets.  Given a subgroup $P\in S$, the \emph{$\Lambda_\cF$-centralizer of $P$} is the set of points $C_{\Lambda_\cF}(P)\subseteq\Lambda_\cF$ satisfying $a\cdot\omega=\omega\cdot a$ for all $a\in P$; a similar definition made for the normalizer $N_{\Lambda_\cF}(P)$.  For $P\leq S$ we also have notions of centralizer and normalizer fusion subsystems, denoted $C_\cF(P)$ and $N_\cF(P)$, which are saturated fusion systems  if $P$ is  fully $\cF$-normalized (i.e., $\abs{N_S(P)}\geq \abs{N_S(\varphi P)}$ for all $\cF$-morphisms $\varphi\colon P\to S$).  We show

\begin{theorem}[Theorem \ref{K-normalizers of minimal are minimal}]\label{thm:IntroKNormalizers}
If $P$ is fully $\cF$-normalized and additionally $C_S(P)\leq P$, then $
C_{\Lambda_\cF}(P)=\Lambda_{C_\cF(P)}$  and  $N_{\Lambda_\cF}(P)=\Lambda_{N_\cF(P)}$.  In other words, the centralizer of $P$ in the minimal $\cF$-characteristic biset is the minimal $C_{\cF}(P)$-characteristic biset, and similarly for normalizers.
\end{theorem}

In fact, we prove a more general statement in terms of Puig's notion of $K$-normalizers.

We interpret these results as saying that the minimal $\cF$-characteristic biset is playing the role of a grouplike object inducing $\cF$ by conjugation, and that we are able to perform many group-theoretic operations in terms of $\Lambda_\cF$.

We close with an open conjecture that the condition $C_S(P)\leq P$ is not necessary in Theorem \ref{thm:IntroKNormalizers}.  In other words, we conjecture that $\cF$-centricity is not an essential concept in the world of minimal characteristic bisets, which would allow us avoid one of the most troublesome technical details in the study of $p$-local finite groups.\\

{\bf Acknowledgements:}  We would like to thank Haynes Miller for his hospitality in hosting the second author at MIT during the Spring of 2013.  It was during this visit that much of the work for this paper was done.

}

\section{Fusion systems}\label{secFusionSystems}

The next few pages contain a very short introduction to fusion systems, which were originally introduced by Puig under the name ``full Frobenius categories,'' cf. \cite{PuigFrobeniusCategories}. The aim is to introduce the terminology from the theory of fusion systems that will be used in the paper, and to establish the relevant notation. For a proper introduction to fusion systems see, for instance, Part I of ``Fusion Systems in Algebra and Topology'' by Aschbacher, Kessar and Oliver, \cite{AKO}.

\begin{definition}
A \emph{fusion system} $\cF$ on a $p$-group $S$, is a category where the objects are the subgroups of $S$, and for all $P,Q\leq S$ the morphisms must satisfy:
\begin{enumerate}
\item Every morphism $\ph\in \Mor_{\cF}(P,Q)$ is an injective group homomorphism, and the composition of morphisms in $\cF$ is just composition of group homomorphisms.
\item $\Hom_{S}(P,Q)\subseteq \Mor_{\cF}(P,Q)$, where
    \[\Hom_S(P,Q) = \{c_s \mid s\in N_S(P,Q)\}\]
    is the set of group homomorphisms $P\to Q$ induced by $S$-conjugation.
\item For every morphism $\ph\in \Mor_{\cF}(P,Q)$, the group isomorphisms $\ph\colon P\to \ph P$ and $\ph^{-1}\colon \ph P \to P$ are elements of $\Mor_{\cF}(P,\ph P)$ and $\Mor_{\cF}(\ph P, P)$ respectively.
\end{enumerate}
We also write $\Hom_{\cF}(P,Q)$ or just $\cF(P,Q)$ for the morphism set $\Mor_{\cF}(P,Q)$; and the group $\cF(P,P)$ of automorphisms is denoted by $\Aut_{\cF}(P)$.
\end{definition}

The canonical example of a fusion system comes from a finite group $G$ with a given $p$-subgroup $S$.
The fusion system of $G$ on $S$, denoted $\cF_S(G)$, is the fusion system on $S$ where the morphisms from $P\leq S$ to $Q\leq S$ are the homomorphisms induced by $G$-conjugation:
\[\Hom_{\cF_S(G)}(P,Q) := \Hom_G(P,Q) = \{c_g \mid g\in N_G(P,Q)\},\]
A particular case is the fusion system $\cF_S(S)$ consisting only of the homomorphisms induced by $S$-conjugation.

Let $\cF$ be an abstract fusion system on $S$. We say that two subgroup $P,Q\leq S$ are\linebreak \emph{$\cF$-con\-ju\-gate}, written $P\sim_{\cF} Q$, if they a isomorphic in $\cF$, i.e., there exists a group isomorphism $\ph\in \cF(P,Q)$.
$\cF$-conjugation is an equivalence relation, and the set of $\cF$-conjugates to $P$ is denoted by $(P)_{\cF}$. The set of all $\cF$-conjugacy classes of subgroups in $S$ is denoted by $Cl(\cF)$.
Similarly, we write $P\sim_S Q$ if $P$ and $Q$ are $S$-conjugate, the $S$-conjugacy class of $P$ is written $(P)_{S}$ or just $[P]$, and we write $Cl(S)$ for the set of $S$-conjugacy classes of subgroups in $S$.
Since all $S$-conjugation maps are in $\cF$, any $\cF$-conjugacy class $(P)_{\cF}$ can be partitioned into disjoint $S$-conjugacy classes of subgroups $Q\in (P)_{\cF}$.

We say that $Q$ is \emph{$\cF$-} or \emph{$S$-subconjugate} to $P$ if $Q$ is respectively $\cF$- or $S$-conjugate to a subgroup of $P$. 
In the case where $\cF = \cF_S(G)$, then $Q$ is $\cF$-subconjugate to $P$ if and only if $Q$ is $G$-conjugate to a subgroup of $P$; in this case the $\cF$-conjugates of $P$ are just those $G$-conjugates of $P$ that are contained in $S$.

A subgroup $P\leq S$ is said to be \emph{fully $\cF$-normalized} if $\abs{N_S P} \geq \abs{N_S Q}$ for all $Q\in (P)_{\cF}$; similarly $P$ is \emph{fully $\cF$-centralized} if $\abs{C_S P} \geq \abs{C_S Q}$ for all $Q\in (P)_{\cF}$.

\begin{definition}
A fusion system $\cF$ on $S$ is said to be \emph{saturated} if the following properties are satisfied for all $P\leq S$:
\begin{enumerate}
\item If $P$ is fully $\cF$-normalized, then $P$ is fully $\cF$-centralized, and $\Aut_S(P)$ is a Sylow $p$-subgroup of $\Aut_{\cF}(P)$.
\item Every homomorphism $\ph\in \cF(P,S)$ with $\ph (P)$ fully $\cF$-centralized extends to a homomorphism $\ph\in \cF(N_\ph,S)$, where
    \[N_\ph:= \{x\in N_S(P) \mid \exists y\in S\colon \ph\circ c_x = c_y\circ \ph\}\]
    is the \emph{extender} of $\varphi$.
\end{enumerate}
\end{definition}
The saturation axioms are a way of emulating the Sylow theorems for finite groups; in particular, whenever $S$ is a Sylow $p$-subgroup of $G$, then the Sylow theorems imply that the induced fusion system $\cF_S(G)$ is saturated (see e.g. \cite[Theorem 2.3]{AKO}).

A particularly important consequence of the saturation axioms, which forms the basis for the key technical Lemma \ref{Sune's lemma}, is as follows:
\begin{lemma}\label{lemNormalizerMap}
Let $\cF$ be saturated. If $P\leq S$ is fully normalized, then for each $Q\in [P]_{\cF}$ there exists a homomorphism $\ph\in \cF(N_S Q, N_S P)$ with $\ph(Q) = P$.
\end{lemma}
For the proof, see Lemma 4.5 of \cite{RobertsShpectorovSaturationAxiom} or Lemma 2.6(c) of \cite{AKO}.

\section{Background on bisets}\label{secBisets}
In this section we recall the basic results about bisets -- finite sets equipped with both a left and a right group action. In addition, we establish the necessary notation relating to bisets.

\begin{definition} Let $G$ and $H$ be finite groups.
A (free) $(G,H)$-biset $\Omega$ is a set endowed with a free left $H$-action and a free right $G$-action, which commute:
\[
h\cdot(\omega\cdot g)=(h\cdot \omega)\cdot g
\]
When it is not clear from context which groups act on $\Omega$, we write $_H \Omega_G$.

Equivalently, $\Omega$ is a left $(H\times G)$-set such that the restrictions of the action to $H\times 1$ and $1\times G$ are free.  This equivalence is formed by setting
\[
(h,g)\cdot \omega=h\cdot\omega\cdot g^{-1}.
\]
Given a $(G,H)$-biset $\Omega$ the \emph{opposite biset} is the $(H,G)$-biset $\Omega^\mathrm{o}$ with the same underlying set and with action defined by
\[
g\cdot \omega^\mathrm{o}\cdot h:=h^{-1}\cdot\omega\cdot g^{-1}.
\]
If $G=H$ and $\Omega\cong\Omega^\mathrm{o}$ as $(G,G)$-bisets, we say $\Omega$ is \emph{symmetric}.

Denote by $A_+(G,H)$ the monoid of isomorphism classes of $(G,H)$-bisets with disjoint union as addition.
If $\Omega\in A_+(G,H)$ and $\Lambda\in A_+(H,K)$, we define the $(G,K)$-biset $\Lambda\circ\Omega$ to be $\Lambda\times_H\Omega$. With $\circ$ as composition, the monoids $A_+(G,H)$ form the morphism sets of a category where the objects are all finite groups. This is also the reason why a $(G,H)$-biset has $G$ acting from the right and not the left, so that the composition order of bisets $\Lambda\circ\Omega$ fits with the general convention for maps and morphisms.

The \emph{point-stabilizer} of an element $\omega$ in a $(G,H)$-biset $\Omega$ is  $\Stab_{H\times G}(\omega)\leq H\times G$, the subgroup consisting of all pairs $(h,g)$ such that $h\cdot \omega=\omega\cdot g$, or equivalently $h\cdot\omega\cdot g^{-1}=\omega$.
A \emph{(injective) $(G,H)$-pair} is a pair $(K,\varphi)$ with $K\leq G$ and $\varphi\colon K\to H$ an injective group map. If $(K,\varphi)$ is a $(G,H)$-pair, denote by $[K,\varphi]$ the $(G,H)$-biset $H\times_{(K,\varphi)}G:= H\x G/(h,kg)\sim(h\ph(k),g)$.
If we also denote by $(K,\varphi)$  the \emph{graph} of $\varphi\colon K\to H$:
\[
(K,\varphi):=\left\{(\varphi(k),k)\in H\times G\right\},
\]
then  $[K,\varphi]\cong (H\times G)/(K,\varphi)$ as $H\times G$-sets.

We will also refer to the graph $(K,\varphi)$ as a \emph{twisted diagonal (subgroup)}.  In the case that $G=H=S$ is a finite $p$-group, $K=P\leq S$, and $\varphi\in\cF(P,S)$ for a given fusion system $\cF$ on $S$, we will refer to $(P,\varphi)$ as an \emph{$\cF$-twisted diagonal (subgroup)}.

The $(G,H)$-pairs $(K,\varphi)$ and $(L,\psi)$ are \emph{$(G,H)$-conjugate} if there are elements $g\in N_G(K,L)$ and $h\in N_H(\varphi(K),\psi(L))$ such that $L=\lc gK$ and
\[
\xymatrix{
K\ar[r]^\varphi\ar[d]_{c_g}&	H\ar[d]^{c_h}\\
L\ar[r]_\psi& H
}
\]
commutes. This happens if and only if the twisted diagonals $(K,\ph)$ and $(L,\psi)$ are conjugate as subgroups of $H\x G$.
\end{definition}

\begin{fact}
The $(G,H)$-bisets $[K,\varphi]$ and $[L,\psi]$ are isomorphic if and only if $(K,\varphi)$ is (G,H)-conjugate to $(L,\psi)$.  Moreover, every transitive $(G,H)$-biset is isomorphic to $[K,\varphi]$ for some $(G,H)$-pair $(K,\varphi)$.  In other words, if $\Omega$ is a transitive $(G,H)$-biset, the stabilizer in $H\times G$ of any point $\omega\in\Omega$ is a subgroup of the form $(K,\varphi)$.
\end{fact}

Let $S$ be a finite $p$-group and $\cF$ a saturated fusion system on $S$.

\begin{definition}
An $(S,S)$-biset $\Omega$ is  \emph{$\cF$-generated} if  all point-stabilizers are $\cF$-twisted diagonal subgroups.

$\Omega$ is \emph{$\cF$-stable} if for every $(S,S)$-pair $(P,\ph)$ and $\cF$-isomorphisms $\eta_1\colon Q \xra{\cong} P$ and $\eta_2\colon \ph P \xra{\cong} R$, we have $\abs[\big]{\Omega^{(P,\varphi)}}=\abs[\big]{\Omega^{(Q,\eta_2\ph\eta_1)}}$.
\end{definition}

\begin{definition}\label{semicharacteristic definition}
An \emph{$\mathcal{F}$-semicharacteristic biset} is an $(S,S)$-biset $\Omega$ that satisfies:
\begin{enumerate}
\item\label{char:Fgen} $\Omega$ is $\cF$-generated.
\item\label{char:Fstable} $\Omega$ is $\cF$-stable. When $\Omega$ is $\cF$-generated, it suffices to check that for each $P\leq S$ and $\varphi\in\cF(P,S)$, we have $\abs[\big]{\Omega^{(P,\varphi)}}=\abs[\big]{\Omega^{(P,\iota_P^S)}}=\abs[\big]{\Omega^{(\varphi(P),\varphi^{-1})}}$, for $\iota_P^S\colon P\to S$ the natural inclusion map.
\end{enumerate}
$\Omega$ is an \emph{$\cF$-characteristic biset} if in addition
\begin{enumerate}\setcounter{enumi}{2}
\item\label{char:ndegen} $\abs{\Omega}/\abs{S}\not\equiv 0 \pmod p$.
\end{enumerate}
\end{definition}

\begin{example}\label{ex:G is characteristic}
Suppose that $S\in \Syl_p(G)$ is equipped with the associated saturated fusion system $\cF := \cF_S(G)$. With left and right multiplication $G$ is the $(S,S)$-biset $_S G_S$, which is always $\cF$-characteristic:

For each $g\in G$, $\Stab_{S\x S}(g)=(S\cap S^g,c_g)$, hence $_S G_S$ is $\cF$-generated. If $c_h\in \cF(P,S)$ us any morphism in $\cF$, the $\cF$-twisted diagonal $(P,c_h)$ is conjugate in $G\x G$ to $(P,\iota_P^S)$ and $(\lc hP, c_h^{-1})$, so $\abs[\big]{\Omega^{(P,c_h)}}=\abs[\big]{\Omega^{(P,\iota_P^S)}}=\abs[\big]{\Omega^{(\lc hP,c_h^{-1})}}$ and $\Omega$ is $\cF$-stable. Finally, $S\in \Syl_p(G)$ implies that $p\nmid \abs{{}_S G_S}/\abs S$.
\end{example}

\subsection{Some fixed point calculations}\label{fixed-point subsection}

In the rest of this section we aim to investigate fixed point sets of the form $[Q,\psi]^{(P,\varphi)}$ that arise in our $\cF$-characteristic bisets.  This will in turn depend on the structure of the transporters $N_{S\times S}((P,\varphi),(Q,\psi))$ via the formula
\[
\abs*{[Q,\psi]^{(P,\varphi)} }=\frac{\abs*{N_{S\times S}((P,\varphi),(Q,\psi))}}{\abs*{(Q,\psi)}}=\frac{\abs*{N_{S\times S}((P,\varphi),(Q,\psi))}}{\abs*{Q}}.
\]
To begin, suppose that $(y,x)\in N_{S\times S}((P,\varphi),(Q,\psi))$, so that for each $p\in P$, we have
\[
(y,x)(\ph(p),p)(y^{-1},x^{-1})=(\psi(q),q)
\]
for some $q\in Q$.  In particular, if $xpx^{-1}=q$, we have $y\ph(p)y^{-1}=\psi(q)=\psi(xpx^{-1})$, so
\[
\xymatrix{
A\ar[r]^\varphi\ar[d]_{c_x}&\varphi A\ar[d]^{c_y}\\
B\ar[r]_\psi&\psi B
}
\]
is a commuting diagram of group homomorphisms with $x\in N_S(A,B)$ and $y\in N_S(\ph A,\psi B)$.  In particular, $\psi\circ c_x\circ\varphi^{-1}= c_y\colon \varphi A\to\psi B$, so that $\psi\circ c_x\circ\varphi^{-1}\in\Hom_S(\varphi A,\psi B)$.

Conversely, consider an element $x\in N_S(A,B)$ with $\eta:=\psi\circ c_x\circ\varphi^{-1}\in\Hom_S(\varphi A,\psi B)$. Then for every element $y\in S$ such that $c_y\big|_{\varphi A}=\eta$, it is easy to see that we have a pair $(y,x)\in N_{S\times S}((P,\varphi),(Q,\psi))$, and that there are $\abs{C_S(\varphi A)}$ such $y$ if there are any.

\begin{definition}
For $A\stackrel\varphi\to\varphi A$ and $B\stackrel\psi\to\psi B$ two morphisms of $\cF$, set
\[
N_{\varphi,\psi}:=\left\{x\in N_S(A,B)\Mid \psi\circ c_x\circ\varphi^{-1}\in\Hom_S(\varphi A,\psi B)\right\}.
\]
Note that the set $N_{\varphi,\psi}$ is independent of the choice of the targets of $\varphi$ and $\psi$, as is $\abs*{[Q,\psi]^{(P,\varphi)}}$.  Since every morphism of $\cF$ factors uniquely as an isomorphism followed by an inclusion, we lose no data by focusing on just the isomorphisms of $\cF$.
\end{definition}

\begin{prop}\label{omnibus transport extender}
Let $P\stackrel\varphi\to\varphi P$ and $Q\stackrel\psi\to\psi Q$ be two isomorphisms of $\cF$.
\begin{enumerate}\renewcommand{\theenumi}{(\alph{enumi})}\renewcommand{\labelenumi}{\theenumi}
\item $N_{\varphi,\psi}=\proj_2\bigl(N_{S\times S}((P,\varphi),(Q,\psi))\bigr)$  and $N_{\varphi^{-1},\psi^{-1}}=\proj_1\bigl(N_{S\times S}((P,\varphi),(Q,\psi))\bigr)$ for $\proj_i$  the \emph{i}th projection $S\x S\to S$, $i=1,2$.
\item  $\abs[\big]{N_{S\times S}((P,\varphi),(Q,\psi))}=\abs[\big]{N_{\varphi,\psi}}\cdot\abs[\big]{C_S(\varphi P)}=\abs[\big]{N_{\varphi^{-1},\psi^{-1}}}\cdot\abs[\big]{C_S(P)}$.
\item  $\abs[\big]{[Q,\psi]^{(P,\varphi)}}=\frac{\abs[\big]{N_{\varphi,\psi}}\cdot\abs[\big]{C_S(\varphi P)}}{\abs*{Q}}=\frac{\abs[\big]{N_{\varphi^{-1},\psi^{-1}}}\cdot\abs[\big]{C_S(P)}}{\abs*{Q}}$.
\item $N_{\varphi,\varphi}=N_\varphi$, the standard extender of $\varphi$.
\item $N_{\varphi,\psi}$ is naturally a free $(N_\varphi,N_\psi)$-biset.
\end{enumerate}
\end{prop}

\begin{proof}
(a)-(d) are immediate from the preceding discussion.  For (e), pick $x\in N_{\varphi,\psi}$, $n\in N_\varphi$, and $m\in N_\psi$.  We have
\begin{eqnarray*}
\psi\circ c_{mxn}\circ\varphi^{-1}&=&(\psi\circ c_m\circ \psi^{-1})\circ(\psi\circ c_x\circ\varphi^{-1})\circ(\varphi\circ c_n\circ\varphi^{-1}) \\
&\in&\Aut_S(\psi B)\circ\Hom_S(\varphi A,\psi B)\circ\Aut_S(\varphi A)\\
&\subseteq&\Hom_S(\varphi A,\psi B),
\end{eqnarray*}
so $m\cdot x\cdot n=mxn\in N_{\varphi,\psi}$.  Freeness is immediate.
\end{proof}

\begin{cor}\label{normalizer size in SxS}
Every $\cF$-twisted diagonal subgroup $(P,\varphi)\in S\times S$  is $(\cF\times\cF)$-isomorphic to some $(Q,\iota_Q^S)$ that is fully $(\cF\times\cF)$-normalized.  Moreover, $(Q,\iota_Q^S)$ is fully $(\cF\times\cF)$-normalized if and only if $Q$ is fully $\cF$-normalized.
\end{cor}

\begin{proof}
That $(P,\varphi)$ is $(\cF\times\cF)$-conjugate with some $(Q,\iota_Q^S)$ is clear from the definition of $\cF\times\cF$.  Proposition \ref{omnibus transport extender} implies that $
\abs*{N_{S\times S}((P,\varphi))}=\abs*{N_\varphi}\cdot \abs*{C_S(\varphi P)}$.  It follows from the definition of the extender that $\abs{N_\varphi}\leq\abs{N_S(P)}$ and $N_{\iota_Q^S}=N_S(Q)$.  Therefore $\abs*{N_{S\times S}((Q,\iota_Q^S))}=\abs*{N_S(Q)}\cdot\abs*{C_S(Q)}$, and this is maximal in the $(\cF\times\cF)$-class of $(P,\varphi)$ precisely when $Q$ is fully $\cF$-normalized (as full $\cF$-normalization implies full $\cF$-centralization).
\end{proof}

%
%
%
%

Our first main goal is to parameterize the semicharacteristic bisets of $\cF$. This will however require a short detour into the realm of sets with only one group action.

\section{The free monoid of $\cF$-sets}\label{sec:monoid of F-sets}
Let $S$ be a finite $p$-group and $\cF$ a saturated fusion system on $S$. In analogy with the finite $G$-sets for a group $G$, this section studies a notion of $\cF$-sets for a fusion system. We give a new proof of \cite[Theorem A]{ReehStableSets}, that every finite $\cF$-set decomposes uniquely, up to $S$-isomorphism, as a disjoint union of irreducible $\cF$-sets. The key lemma is the same as in \cite{ReehStableSets}, but the main part of the proof is different: In the proof below, the decomposition is constructed explicitely by considering the actual $\cF$-sets in play, while \cite{ReehStableSets} relies on the structure of the Burnside ring of $\cF$ and linear algebra.

\begin{definition}
A finite \emph{$\cF$-stable} $S$-set, or just \emph{$\cF$-set}, is a finite set $X$ with an action of $S$ such that for all $P\leq S$ and $\varphi\in\cF(P,S)$ the order of the fixed point sets of $P$ and $\varphi P$ are equal:  $\abs{X^P}=\abs{X^{\varphi P}}$.

Let $A_+(S)$ be the free commutative monoid of isomorphism classes of finite $S$-sets with disjoint union as addition, and let $A_+(\cF)\subseteq A_+(S)$ be the submonoid of isomorphism classes of $\cF$-sets. Both $A_+(S)$ and $A_+(\cF)$ are semirings with Cartesian product as multiplication. Our goal in this section is to show that $A_+(\cF)$ is a free commutative monoid.
\end{definition}

\begin{definition}
The $S$-set $X$ is \emph{$\cF$-stable above level $n$} if for any $P\leq S$ with $\abs{P}\geq p^n$ and $\varphi\in\cF(P,S)$, we have $\abs{X^P}=\abs{X^{\varphi P}}$.
Clearly an $S$-set $X$ is an $\cF$-set if and only if $X$ is $\cF$-stable above level 0.
\end{definition}

The following is the main technical result that implies the freeness of $A_+(\cF)$.  We do not repeat the proof, but we do recall how it gives rise to an additive basis in the following.

\begin{lemma}[\cite{ReehStableSets}, Lemma 4.7]\label{Sune's lemma}
Suppose that $X$ is an $S$-set that is $\cF$-stable above level $n+1$ and that the order of every stabilizer of every element of $X$ is at least $p^{n+1}$.  If $P,Q\leq S$ are $\cF$-conjugate subgroups of order $p^n$ and $Q$ is fully normalized in $\cF$, then $\abs{X^Q}\geq\abs{X^P}$.
\end{lemma}

\begin{notation}
Denote by $Cl(S)$ the set of $S$-conjugacy classes of subgroups of $S$, and by $Cl(\cF)$ the set of $\cF$-conjugacy classes of subgroups.  A class in $Cl(S)$ will be denoted $(P)_S$, and a class in $Cl(\cF)$ will be $(P)_\cF$.  Also, for $(P)_S\in Cl(S)$, let $[P]$ denote the isomorphism class of the $S$-set $S/P$.
\end{notation}

We now construct a collection of $\cF$-sets satisfying particular structural properties. We will later show, in Corollary \ref{F-basis cor}, that such $\cF$-sets are irreducible and form a basis for $A_+(\cF)$.
\begin{theorem}\label{F-basis}
For each $P\leq S$ fully normalized in $\cF$, there is an $\cF$-set
\[
X_P=\coprod_{(Q)_S\in Cl(S)} c_Q\cdot[Q],
\]
for $c_Q\in\ZZ_{\geq0}$, that is uniquely determined as an $S$-set by requiring
\begin{enumerate}
\item\label{itm:cP1} $c_P=1$,
\item\label{itm:nonzerocoeff} If $Q$ is fully normalized and $c_Q\neq 0$, then $Q\cong_\cF P$.
\end{enumerate}
\end{theorem}

\begin{remark}\label{rmk:additional properties}
The particular sets that we construct in the proof have additional properties:
\begin{enumerate}\setcounter{enumi}{2}
\item\label{itm:subconjugateorbits} If $c_Q\neq 0$, $Q$ is $\cF$-subconjugate to $P$.
\item\label{itm:cPcQ1} If $P\cong_\cF Q$ are both fully normalized, then $X_P=X_Q$, which contains exactly one copy of each orbit $[P]$ and $[Q]$.
\end{enumerate}

In Corollary \ref{cor:minimal Fsets are unique}, we argue that $X_P$ in Theorem \ref{F-basis} is actually uniquely determined by properties \ref{itm:cP1} and \ref{itm:nonzerocoeff}. Therefore $X_P$ must have the  structure specified in the proof below and satisfies \ref{itm:subconjugateorbits} and \ref{itm:cPcQ1}.

Finally, we should note that while only \ref{itm:cP1}-\ref{itm:cPcQ1} will be used in this paper, much more can be said about the coefficients $c_Q$ and the $Q$-fixed-point orders of $X_P$.  The computations involved relate the combinatorics of the poset of subgroups of $S$ to the shape of the category $\cF$ (i.e., which subgroups are made conjugate in the fusion system) together with $p$-local data concerning the orders of normalizers of certain subgroups.  See \cite{GelvinReehYalcinBasisOfBurnsideRing} for more details.
\end{remark}

\begin{proof}
We will begin with the $S$-set $[P]$ and construct, in a minimal way, an $\cF$-set containing $[P]$. We proceed level by level using Lemma \ref{Sune's lemma} until we have a set which is $\cF$-stable above level $0$ and hence an $\cF$-set.

Suppose that $\abs{P}=p^n$.  If $Q\cong_\cF P$ but $Q\not\cong_S P$, $[P]$ will not be $\cF$-stable above level $n$:  $\abs{[P]^P}=\abs{N_S(P)}/\abs{P}$ but $\abs{[P]^Q}=0$.  To correct this while respecting \ref{itm:subconjugateorbits}, we must add some number of copies of $[Q]$.  Since $\abs{[Q]^Q}=\abs{N_S(Q)}/\abs{Q}$ and $\abs{Q}=\abs{P}$, it is easy to see that we must add $\frac{\abs{N_S(P)}}{\abs{N_S(Q)}}$ copies of $[Q]$ so that the number of $Q$-fixed points of the resulting $S$-set equals the number of $P$-fixed points.  It follows easily that, if $P=Q_1,Q_2,\dotsc,Q_a$ are representatives of the $S$-conjugacy classes of the $\cF$-conjugacy class $(P)_\cF$, the $S$-set
\[
X_P^{(n)}:=\coprod_{i=1}^a\frac{\abs{N_S(P)}}{\abs{N_S(Q)}}\cdot[Q]
\]
is an $S$-set, $\cF$-stable above level $n$, that satisfies \ref{itm:cP1}-\ref{itm:subconjugateorbits}. Note that had we used another fully normalized subgroup $Q\cong_\cF P$ instead of $P$, we would arrive at the same set: $X_Q^{(n)}=X_P^{(n)}$. Because the construction only depends on $X_P^{(n)}$,  $X_Q=X_P$ and \ref{itm:cPcQ1} follows.

The trick then is to show that  $X_P^{(n)}$ is contained in an $S$-set $X_P^{(n-1)}$ that satisfies \ref{itm:cP1}-\ref{itm:subconjugateorbits} and is $\cF$-stable above level $n-1$; the rest follows by obvious induction.  So, suppose that $Q\leq S$ is a subgroup of order $p^{n-1}$, and let $R\in (Q)_\cF$ be a fully normalized representative fro the $\cF$-conjugacy class.  Lemma \ref{Sune's lemma} implies that
\[
\abs*{(X_P^{(n)})^Q}\leq\abs*{(X_P^{(n)})^R}.
\]
The claim is that if the inequality is proper, we can add a certain number of copies of $[Q]$ to $X_P^{(n)}$ to force equality.  Let $\varphi\in\cF(N_S(Q),N_S(R))$ be such that $\varphi(Q)=R$; this exists by the saturation of $\cF$ and the assumption that $R$ is fully $\cF$-normalized.  $W_S(Q):=N_S(Q)/Q$ naturally acts on $(X_P^{(n)})^Q$.  Similarly $W_S(R)$ naturally acts on $(X_P^{(n)})^R$, and $\varphi$ induces a map $W_S(Q)\to W_S(R)$ and thus an action of $W_S(Q)$ on $(X_P^{(n)})^R$.

Decompose
\[(X_P^{(n)})^R=(X_P^{(n)})^R_{f}\amalg(X_P^{(n)})^R_{nf},\]
where $(X_P^{(n)})^R_f$ is the subset of elements on which $W_S(Q)$ acts freely and $(X_P^{(n)})^R_{nf}$ are those elements on which $W_S(Q)$ does not act freely.  In other words, $\omega\in(X_P^{(n)})^R_{nf}$ iff $\omega\in(X_P^{(n)})^R$ and $\Stab_{W_S(Q)}(\omega)\neq 1$.  Similarly, decompose
\[(X_P^{(n)})^Q=(X_P^{(n)})^Q_f\amalg(X_P^{(n)})^Q_{nf}.\]
If $\omega\in(X_P^{(n)})^Q_{nf}$, let $\overline A\leq W_S(Q)$ be the (nontrivial) stabilizer of $\omega$ in $W_S(Q)$, and $A\leq N_S(Q)$ the preimage of $\overline{A}$.  Clearly $A\leq \Stab_S(\omega)$, and $\abs{A}\geq p^n$.  In other words, every element of $(X_P^{(n)})^Q_{nf}$ lies in $(X_P^{(n)})^A$ for some $A$ of order strictly greater than that of $Q$; the same statement holds for $(X_P^{(n)})^R_{nf}$.  By the inductive hypothesis, $\abs[\big]{(X_P^{(n)})^A}=\abs[\big]{(X_P^{(n)})^{\varphi(A)}}$ for all such $A$, so we conclude
\[
\abs*{(X_P^{(n)})^Q_{nf}}=\abs*{(X_P^{(n)})^R_{nf}}
\]
by the same inclusion-exclusion argument in the proof of Lemma \ref{Sune's lemma}.  Thus $\abs[\big]{(X_P^{(n)})^R}-\abs[\big]{(X_P^{(n)})^Q}=\abs[\big]{(X_P^{(n)})^R_f}-\abs[\big]{(X_P^{(n)})^Q_f}$, so in particular \[c_Q:=\frac{\abs[\big]{(X_P^{(n)})^R}-\abs[\big]{(X_P^{(n)})^Q}}{\abs{ W_S(Q)}}\in\ZZ_{\geq0}.\]
This can be done for all subgroups $Q\leq S$ of order $p^{n-1}$, with chosen representatives for each $\cF$-conjugacy class.

From here it is easy to see that if we set
\[
X_P^{(n-1)}=X_P^{(n)}\amalg\coprod_{\substack{(Q)_S\in Cl(S),\\\text{s.t.\,} \abs Q = p^{n-1}}} c_Q\cdot[Q],
\]
 then $X_P^{(n-1)}$ satisfies \ref{itm:cP1}-\ref{itm:subconjugateorbits} and is $\cF$-stable above level $n-1$, so we're done.
\end{proof}

\begin{cor}\label{F-basis cor}
Choose a fully normalized representative $P^*\in (P)_\cF$ from each class in $Cl(\cF)$. The $\cF$-sets $\{X_{P^*}\mid (P^*)_\cF\in Cl(\cF)\}$ then form a basis for $A_+(\cF)$.
\end{cor}

\begin{proof}
Conditions \ref{itm:cP1} and \ref{itm:nonzerocoeff} imply that there can be no non-trivial $\ZZ_{\geq0}$-linear (indeed, $\ZZ$-linear) relations amongst the $X_{P^*}$, so it suffices to show that every $\cF$-set can be written as a sum of these.

Let $X$ be an arbitrary $\cF$-set, and pick a decomposition
\[X=\coprod_{(P)_S\in Cl(S)}c_P\cdot[P].\]
Consider the chosen representative $P^*\in (P)_\cF$ for each $P\leq S$, and set
\[
Y:=\coprod_{P^*} c_{P^*}\cdot[X_{P^*}].
\]
Consider $X-Y\in A(S)$, in the Grothendieck group of $A_+(S)$; if this can be shown to be 0, $X$ will lie in $\Span_{\ZZ_{\geq 0}}\{X_{P^*}\mid (P^*)_\cF\}$, and we're done. We can extend $\abs{X^Q}$ linearly to the formal differences in $A(S)$ in order to count generalized fixed points.
If $X-Y\neq 0$, there is some  subgroup $Q\leq S$ of maximal order such that $c_Q(X-Y)\neq 0$.
But for $Q^*$ the chosen fully $\cF$-normalized representative of $(Q)_\cF$, we have $c_{Q^*}(X-Y)=0$ by construction, so
\[
\abs[\big]{(X-Y)^Q}=c_Q(X-Y)\cdot \abs{W_S(Q)}\neq 0,\quad\textrm{while}\quad
\abs[\big]{(X-Y)^{Q^*}}=c_{Q^*}(X-Y)\cdot\abs{W_S(Q)}=0.
\]
Hence $\abs{X^{Q^*}}=\abs{Y^{Q^*}} = \abs{Y^Q}\neq\abs{X^Q}$ contradicting $\cF$-stability of $X$.
\end{proof}

\begin{cor}\label{cor:minimal Fsets are unique}
Suppose  $P\leq S$ is fully normalized. The $\cF$-set $X_P$ is uniquely determined by properties \ref{itm:cP1} and \ref{itm:nonzerocoeff}, and is the unique minimal $\cF$-set containing $[P]$ as an orbit.

By Remark \ref{rmk:additional properties}, it then follows that $X_P$ depends only on the class $(P)_{\cF}$, and for each fully normalized $Q\in (P)_\cF$ the $\cF$-set $X_P$ contains the orbit $[Q]$ exactly once.
\end{cor}

\begin{proof}
$X_P$ is part of a basis for $A_+(\cF)$ as in Corollary \ref{F-basis cor}. By properties \ref{itm:cP1} and \ref{itm:nonzerocoeff} $X_P$ is the only basis element that contains $[P]$ as an orbit, so every $\cF$-set containing $[P]$ has to contain a copy of the basis element $X_P$. It follows that $X_P$ is the unique smallest $\cF$-set containing $[P]$.
\end{proof}

This ends our detour to sets with only one group action, and we return to the world of bisets, in particular the $\cF$-semicharacteristic ones.

\section{The parameterization of semicharacteristic bisets of $\cF$}\label{sec:parameterization of semicharacteristic bisets}
In this section Theorem \ref{parameterization of semicharacteristic bisets} parameterizes all the semicharacteristic bisets of $\cF$. The method of approach is to apply the structure results of section \ref{sec:monoid of F-sets} to the product fusion system $\cF\x \cF$ and the monoid of $(\cF\x\cF)$-sets.

\begin{lemma}\label{twisted diagonal subgroups are conjugate to diagonal subgroups}
Let $(P,\varphi)$ and $(Q,\psi)$ be two twisted diagonal subgroups of $S\times S$.  Then $(P,\varphi)\cong_{\cF\times\cF}(Q,\psi)$ if and only if there exist $\cF$-iso\-mor\-phisms $\eta_1\in\cF(P,Q)$ and $\eta_2\in\cF(\varphi P,\psi Q)$ such that
\[
\xymatrix{
P\ar[r]^{\eta_1}_\cong\ar[d]_\varphi&Q\ar[d]^\psi\\
\varphi P\ar[r]^{\eta_2}_\cong&\psi Q
}
\]
commutes.  In particular, any twisted diagonal subgroup $(P,\varphi)\leq S\times S$ with $\varphi\in\cF(P,S)$ is $(\cF\x\cF)$-isomorphic to every $(Q,\iota_Q^S)$ where $Q\cong_\cF P$.
\end{lemma}

\begin{proof}
Obvious from the definition of $\cF\times \cF$.
\end{proof}

\begin{prop}\label{prop:FFstable}
A (free) $(S,S)$-biset $\Omega$ is $\cF$-stable if and only if $\Omega$ is $(\cF\times\cF)$-stable when viewed as an $(S\times S)$-set.
\end{prop}

\begin{proof}
A morphism of $\cF\times\cF$ is the restriction of a morphism $(\varphi,\psi)$, for $\varphi\in\cF(P,S)$ and $\psi\in\cF(Q,S)$, to some subgroup of $P\times Q$.  As $\Omega$ is bifree, the only subgroups of $S\times S$ with nonempty fixed point sets are twisted diagonals $(P,\varphi)$. By Lemma \ref{twisted diagonal subgroups are conjugate to diagonal subgroups} $(P,\varphi)\cong_{\cF\times\cF} (Q,\psi)$ iff there exist $\cF$-isomorphisms $\eta_1\colon Q\xra{\cong} P$ and $\eta_2\colon \ph P \xra{\cong} \psi Q$ such that $\psi=\eta_2\ph\eta_1$. Hence the $(\cF\times\cF)$-stability condition is equivalent to the condition for $\cF$-stable bisets.
\end{proof}

\begin{theorem}\label{parameterization of semicharacteristic bisets}
Let $\cF$ be a saturated fusion system on $S$.  For each $\cF$-conjugacy class of subgroups $(P)_\cF\in Cl(\cF)$ there is an associated $\cF$-semicharacteristic biset $\Omega_P$: Supposing $P$ is fully normalized, $\Omega_P$ is the smallest $\cF$-semicharacteristic biset containing $[P,\iota_P^S]$. The sets $\Omega_P$, taken together, form an additive basis for the free monoid of semicharacteristic bisets of $\cF$.  Moreover, an $\cF$-semicharacteristic biset
\[
\Omega=\coprod_{(P)_\cF\in Cl(\cF)} c_P\cdot\Omega_P
\]
is  $\cF$-characteristic  if and only if $p\nmid c_S$.
\end{theorem}

\begin{proof}
Pick a representative $P\in(P)_\cF$ such that $(P,\iota_P^S)$ is fully normalized in $\cF\times\cF$; we can choose such a $P$ by Corollary \ref{normalizer size in SxS}.  Define $\Omega_P$ to be the unique $(\cF\times\cF)$-set corresponding to the subgroup $(P,\iota_P^S)\leq S\times S$ defined in Theorem \ref{F-basis}, and by Corollary \ref{cor:minimal Fsets are unique} this is the smallest $(\cF\x\cF)$-set containing $[P,\iota_P^S]$. Property \ref{itm:subconjugateorbits} of remark \ref{rmk:additional properties} states that every point-stabilizer of $\Omega_P$ is $(\cF\x \cF)$-subconjugate to the diagonal $(P,\iota_P^S)$, so $\Omega_P$ is $\cF$-generated and hence semicharacteristic for $\cF$.

The collection $\left\{\Omega_P\right\}_{(P)_\cF\in Cl(\cF)}$ forms a basis for a submonoid of $A_+(\cF\times\cF)$, as it is part of the basis for the entire monoid  $A_+(\cF\times\cF)$.  The submonoid spanned by the $\Omega_P$ consists only of those $(\cF\times\cF)$-sets whose point-stabilizers are $\cF$-twisted diagonal subgroups.  By the same downward induction in the proof of Corollary \ref{F-basis cor}, we see that every $(\cF\times\cF)$-set with point-stabilizers $\cF$-twisted diagonal subgroups lies in this submonoid.  Finally, being $\cF$-semicharacteristic is equivalent to having $\cF$-twisted diagonal point-stabilizers and being $(\cF\times\cF)$-stable (Proposition \ref{prop:FFstable}), thus proving that the $\Omega_P$ form a basis for the monoid of semicharacteristic bisets of $\cF$.

To prove the last claim, it is enough to show that $p$ divides $\abs{\Omega_P/S}=\abs{\Omega_P}/\abs{S}$ if and only $P\neq S$.  As $\abs{[P,\varphi]} = \abs{S\times S}/\abs{P}$, it is clear that $p$ divides $\abs{[P,\varphi]}/\abs{S}$ if and only if $\abs{P}<\abs{S}$.  As every point-stabilizer of $\Omega_P$ is $(\cF\x\cF)$-subconjugate to $[P,\iota_P^S]$, it follows that $\abs{\Omega_P}$ is divisible by $\abs{[P,\iota_P^S]}/\abs{S}$ which is divisible by $p$ if $P\neq S$.  Therefore the choice of the number $c_P$ has no effect on whether or not $\Omega$ is $\cF$-characteristic when $P\neq S$.

Finally, $\Omega_S$ can be decomposed
\[
\Omega_S=\left(\coprod_{[\alpha]\in\Out_\cF(S)}[S,\alpha]\right)\amalg\left(
\coprod_{\substack{\abs{P}<\abs{S}\\ \varphi\in\cF(P,S)}}c_{P,\varphi}[P,\varphi]\right)
\]
for constants $c_{P,\varphi}\in\ZZ_{\geq 0}$.  Each term $[S,\alpha]$ has $\abs{S}$ elements, while $p\abs S\Mid \abs{[P,\varphi]}$ when $\abs{P}<\abs{S}$.  Therefore $\abs{\Omega_S}/\abs{S}\equiv\abs{\Out_\cF(S)}\not\equiv 0$ modulo $p$ by the saturation axioms of fusion systems.
\end{proof}

\begin{cor}\label{minimal characteristic biset}
Each fusion system has a unique minimal $\cF$-characteristic biset $\Lambda=\Lambda_\cF$, in the sense that if $\Omega$ is any $\cF$-characteristic biset for $\cF$, up to isomorphism we have $\Lambda\subseteq\Omega$.
\end{cor}

\begin{proof}
Define $\Lambda_\cF=\Omega_S$ in the notation of  Theorem \ref{parameterization of semicharacteristic bisets}; the rest is immediate.
\end{proof}

\begin{prop}
Each of the $\cF$-semicharacteritic basis elements $\Omega_P$ is a symmetric $(S,S)$-biset. Hence every $\cF$-semicharacteristic biset is symmetric.
\end{prop}

\begin{proof}
$\Omega_P^\mathrm{o}$ is $\cF$-semicharacteristic and contains the orbit $[P,\iota_P^S]^\mathrm{o}\cong [P,\iota_P^S]$. Because $\Omega_P$ is the smallest $\cF$-semicharacteristic biset containing $[P,\iota_P^S]$, we must have $\Omega_P\subseteq \Omega_P^\mathrm{o}$. Size considerations, or applying $(-)^\mathrm{o}$ again, tell us that equality $\Omega_P=\Omega_P^\mathrm{o}$ holds.
\end{proof}

\section{Minimal characteristic bisets of constrained fusion systems}\label{sec:ModelsOfConstrainedFusionSystems}
We know that any finite group $G$ is a $\cF$-characteristic biset for its associated fusion system $\cF_S(G)$; see example \ref{ex:G is characteristic}. For a constrained fusion system $\cF$, a saturated fusion system that contains a normal and $\cF$-centric subgroup, Broto-Castellana-Grodal-Levi-Oliver have shown that $\cF$ has a unique minimal group \emph{model}. This section shows that the model for a contained fusion system is not just a $\cF$-characteristic biset, it is always isomorphic to the minimal $\cF$-characteristic biset for the fusion system.

\begin{prop}\label{stabilizers contain normal p-subgroups}
Let $G$ be a finite group with $\cF=\cF_S(G)$ and $N\leq S$ a normal subgroup of $G$.  If $(P,\varphi)$ is a point-stabilizer of the $(S,S)$-biset $_SG_S$, then $N\leq P$.
\end{prop}

\begin{proof}
Pick $g\in G$ and suppose that $(Q,\psi)$ stabilizes $g$, so that $g\cdot q=\psi(q)\cdot g$ for all $q\in Q$.  Therefore $\psi(q)=gqg^{-1}$ and $g\in N_G(Q,S)$.  As $N\trianglelefteq G$, we have $g\in N_G(N\cdot Q,N\cdot S)$.  As $N\leq S$, if we set $P=N\cdot Q$ we have that conjugation by $g$ induces a map $\varphi\in\cF(P,S)$.  Thus $gpg^{-1}=\varphi(p)$ for all $p\in P$, or $g\cdot p=\varphi(p)\cdot g$. Thus $g\in({}_SG_S)^{(P,\varphi)}$ and $(P,\varphi)\leq \Stab_{S\times S}(g)$.  The result follows.
\end{proof}

Note that in Proposition \ref{stabilizers contain normal p-subgroups}, we do not assume that $S\in\Syl_p(G)$, only that  $S$ contains a normal $p$-subgroup of $G$.  If we additionally require that $S$ is Sylow in $G$, there is a canonical choice for $N\trianglelefteq G$, namely the largest normal $p$-subgroup of $G$.

\begin{notation}
If $G$ is a finite group, $O_p(G)$ denotes the largest normal $p$-subgroup of $G$, and $O_{p'}(G)$ the largest normal $p'$-subgroup.
\end{notation}

\begin{cor}\label{stabilizers of Sylow inclusions contain O_p(G)}
Let $G$ be a finite group with $S\in\Syl_p (G)$ and $\cF=\cF_S(G)$. If the $\cF$-characteristic biset $_S G_S$ decomposes as $\coprod\limits_{(P)_\cF\in Cl(\cF)}c_P\cdot\Omega_P$, then  $c_P\neq 0$ implies $O_p(G)\leq P$.
\end{cor}

\begin{proof}
By Proposition \ref{stabilizers contain normal p-subgroups} and the fact that $O_p(G)=\bigcap\limits_{S'\in\Syl_p(G)} S'$, we see that every every point-stabilizer of $_SG_S$ is of the form $(P,\varphi)$ with $O_p(G)\leq P$.  As the $\cF$-semicharacteristic biset $\Omega_Q$ contains the $(S,S)$-biset $[Q,\iota_Q^S]$, which has an element with stabilizer $(Q,\iota_Q^S)$, it follows that $c_Q=0$ for all $Q\not\geq O_p(G)$.  The result follows.
\end{proof}
There is a general version of Proposition \ref{stabilizers contain normal p-subgroups} and Corollary \ref{stabilizers of Sylow inclusions contain O_p(G)} for abstract fusion systems (Proposition \ref{prop:OpF contained in stabilizers}),  but the proof is   more involved.

\begin{definition}
Let $G$ be a finite group.
\begin{itemize}
\item $G$ is \emph{$p'$-reduced} if $O_{p'}(G)=1$.
\item If $G$ is $p'$-reduced, $G$ is \emph{$p$-constrained} if $C_G(O_p(G))\leq O_p(G)$.
\end{itemize}
Note that $G/O_{p'}(G)$ is always $p'$-reduced, so that we might define a general $G$ to be $p$-constrained if $G/O_p'(G)$ is $p$-constrained.  We will not make use of this definition here.
\end{definition}

\begin{definition}
Let $\cF$ be a saturated fusion system on $S$. We write $O_p(\cF)$ for the largest normal subgroup of $\cF$.  Thus, $O_p(\cF)\trianglelefteq S$ is maximal subject to the requirement that for every $\varphi\in\cF(P,Q)$, there is some extension $\widetilde\varphi\in\cF(P\cdot O_p(\cF),Q\cdot O_p(\cF))$ such that $\widetilde\varphi(O_p(\cF))=O_p(\cF)$.

A saturated fusion system $\cF$ is \emph{constrained} if $O_p(\cF)$ is $\cF$-centric, or equivalently if $C_S(O_p(\cF))\leq O_p(\cF)$.
A \emph{model} for the constrained fusion system $\cF$ is a finite group $M$ that is $p'$-reduced, $p$-constrained, contains $S$ as a Sylow $p$-subgroup, and  $\cF=\cF_S(M)$.
\end{definition}

\begin{theorem}[{\cite[Proposition C]{BCGLO1}}]
Every constrained fusion system has a unique model.
\end{theorem}

We then reach the main result of this section describing the model of a constrained fusion system as a $\cF$-characteristic biset.
\begin{theorem}\label{models are minimal characteristic bisets}
Let $\cF$ be a constrained fusion system on $S$ and $M$ the model for $\cF$.  Then the $(S,S)$-biset $_SM_S$ is the unique minimal $\cF$-characteristic biset $\Lambda_\cF$ of $\cF$.
\end{theorem}

\begin{proof}
We will show (1) if $(P,\iota_P^S)$ is a point-stabilizer of $_SM_S$, then $P=S$, and (2) any two elements of $_SM_S$ whose stablizers are $(S,\id)$ lie in the same $(S,S)$-orbit. In light of the characterization of basis element of $A_+(\cF\times\cF)$ from Theorem \ref{F-basis}, the result will follow immediately from these facts and Theorem \ref{parameterization of semicharacteristic bisets}: (1) shows that $_S M_S$ is a multiple of $\Omega_S=\Lambda_\cF$, and (2) shows that $_S M_S$ contains at most one copy of $\Omega_S$.

(1):  Pick $m\in{}_SM_S$,  $\Stab_{S\x S}(m)=(P,\iota_P^S)$. By Proposition \ref{stabilizers contain normal p-subgroups}, we may assume that $O_p(G)\leq P$.  Thus for any $m\in ({}_SM_S)^{(P,\iota_P^S)}$, we have $m\cdot a = a\cdot m$ for all $a\in P$.  Therefore $m\in C_M(P)\leq C_M(O_p(G))\leq O_p(G)\leq P$, so that $m\in S$ and $m$ induces the automorphism $c_{m}\in\Inn(S)$.  Thus for all $s\in S$, $m\cdot s=c_{m}(s)\cdot m$, or $m\in({}_SM_S)^{(S,c_{m})}$.  As $(P,\iota_P^S)$ was already identified as the stabilizer of $m$, we conclude $P=S$ and $m\in Z(S)$.

(2):  Suppose that $m,m'\in{}_SM_S$ are two elements with point-stabilizer $(S,\id)$.  By the last conclusion of part (1), we have $m,m'\in Z(S)\leq S$, and as $_SS_S$ is a transitive subbiset of $_SM_S$, the result follows.
\end{proof}

\section{Centric minimal characteristic bisets arising from linking systems}\label{sec:LinkingSystemsAsCharacteristicBisets}
In this section we describe the relationship between a centric linking system $\cL$ for a saturated fusion system  and the minimal $\cF$-characteristic biset.

For  $\cF$-centric subgroups $P,Q\leq S$, identify $N_S(P,Q)$ with its image in $\cL(P,Q)$.  The composite of $\fg\in\cL(P,Q)$ and $\fh\in\cL(Q,R)$ will be written $\fh\cdot\fg\in\cL(P,R)$.

We recall the extension result for morphisms of linking systems:

\begin{theorem}[\cite{OliverVenturaTransporterSystems}]\label{extension in linking systems}
Pick $\fg\in\cL_{\iso}(P,Q)$  and normal supergroups $P\trianglelefteq\widetilde P$, $Q\trianglelefteq\widetilde Q$.  If for every $\widetilde p\in\widetilde P$ we have $\fg\cdot\widetilde p\cdot\fg^{-1}\in\widetilde Q$, then $\fg$ has  a unique extension $\widetilde\fg\in\cL(\widetilde P,\widetilde Q)$.
\end{theorem}

\begin{cor}\label{conditions for nonextendability in L}
Let $\fg\in\cL_\iso(P,Q)$ be an isomorphism of $\cL$.  The following are equivalent:
\begin{itemize}
\item[(a)] $\fg$ is nonextendable.
\item[(b)] $\left(\fg^{-1} \cdot N_S(Q)\cdot\fg\right)\cap N_S(P)=P$.
\item[(c)] $\left(\fg\cdot N_S(P)\cdot\fg^{-1}\right)\cap N_S(Q)=Q$.
\end{itemize}
\end{cor}

\begin{proof}
(a)$\Leftrightarrow$(b):  $\fg$ can always extend to $\left(\fg^{-1} \cdot N_S(Q)\cdot\fg\right)\cap N_S(P)$ by Theorem \ref{extension in linking systems}.  On the other hand, if $\fg$ is extendable, without loss of generality we may assume that $\fg$ extends to some $\widetilde\fg\in\cL_\iso(\widetilde P,\widetilde Q)$ with $P\trianglelefteq \widetilde P$.  Then for any $\widetilde p\in\widetilde P$, the diagram
\[
\xymatrix{
\widetilde P\ar[r]^-{\widetilde\fg}\ar[d]_{\widetilde p}&\widetilde Q\ar[d]^{c_{\widetilde\fg}(\widetilde p)}\\
\widetilde P\ar[r]_-{\widetilde\fg}&\widetilde Q
}
\]
commutes in $\cL$.  Here $c_{\widetilde\fg}\in\cF(\widetilde P,\widetilde Q)$ is the image of $\widetilde\fg$ in the underlying fusion system.  On restriction, this diagram becomes
\[
\xymatrix{
P\ar[r]^-\fg\ar[d]_-{\widetilde p}&Q\ar[d]^{c_{\widetilde\fg}(\widetilde p)}\\
P\ar[r]_-\fg&Q
}.
\]
Thus $\fg^{-1}\cdot c_{\widetilde\fg}(\widetilde p)\cdot \fg=\widetilde p\in \left(\fg^{-1} \cdot N_S(Q)\cdot\fg\right)\cap N_S(P)$, and the result follows.

(a)$\Leftrightarrow$(c):  If $\widetilde\fg$ is an extension of $\fg$, then $\widetilde\fg^{-1}$ is an extension of $\fg^{-1}$.  Thus the equivalence of (a) and (c) is the same as that of (a) and (b), with $\fg^{-1}$ in the role of $\fg$.
\end{proof}

One can use this result to prove that the equivalence relation on the set of isomorphisms of $\cL$ generated by restriction has a particularly nice structure.

\begin{theorem}[\cite{ChermakFusionSystemsAndLocalities}, Lemma A.8]\label{thm:unique extensions in linking system}
Let $\fg_1\in\cL_\iso(P_1,Q_1)$ and $\fg_2\in\cL_\iso(P_2,Q_2)$ be two isomorphisms that can be connected by a chain of extensions and restrictions.  Then there is an isomorphism $\fh$ with source containing $\langle P_1,P_2\rangle$ and target containing $\langle Q_1,Q_2\rangle$ such that the restriction of $\fh$ to $P_i$ is $\fg_i$, $i=1,2$.

In particular, each equivalence class of isomorphisms of $\cL$ contains a unique maximal element $\fk$, in the sense that every element of that class is a restriction of $\fk$.  This unique maximal element is of necessity nonextendable, and each nonextendable isomorphism appears as the maximal element of a different class.
\end{theorem}

\begin{notation}
Let $\fI$ denote the set of nonextendable isomorphisms of $\cL$. By Theorem \ref{thm:unique extensions in linking system} every morphism of $\cL$ is then the restriction of a unique isomorphism in $\fI$.

($\fI$ is in fact the underlying set of Chermak's partial group version of a linking system.)
\end{notation}

\begin{lemma}
The set $\fI$ carries a natural $(S,S)$-biset structure.
\end{lemma}

\begin{proof}
Pick $P\xra{\fg} Q\in\fI$ and $a,b\in S$.  Define $a\cdot\fg\cdot b\in\cL(\lc{b^{-1}}P,\lc aQ)$ to be the composite
\[
\xymatrix{
\lc{b^{-1}}P\ar[r]^-b&P\ar[r]^-\fg& Q\ar[r]^-a&\lc aQ.
}
\]
Pick some $n\in N_S(\lc aQ)$ such that $(a\fg b)^{-1}\cdot n\cdot (a\fg b)\in N_S(\lc{b^{-1}}P)$, then $\fg^{-1}(a^{-1} n a)\fg\in N_S(P)$.  As $\fg\in\fI$ is nonextendable, Corollary \ref{conditions for nonextendability in L} forces $a^{-1}na\in Q$, so $n\in\lc aQ$ and $a\cdot\fg\cdot b$ is nonextendable.
\end{proof}

It is not the case that $_S\fI_S$ is an $\cF$-characteristic set, as the example of $\cF_{D_8}(A_6)$ demonstrates. The main failing is that the elements of $\fI$, being morphisms in $\cL$, only see the $\cF$-centric subgroups.

\begin{example}
Inside $\cF:=\cF_{D_8}(A_6)$, the Sylow $2$-subgroup $D_8$ has the following subgroup diagram:
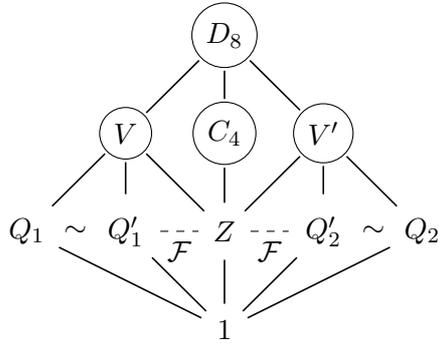
\[
\begin{tikzpicture}[node distance=1.3cm]
\begin{scope}[circle, inner sep=.1cm]
\node[draw] (D8) {$D_8$};
\node[draw] (C4) [below of=D8] {$C_4$};
\node[draw] (V) [left of=C4] {$V$};
\node[draw] (V') [right of=C4] {$V'$};
\node (Z) [below of=C4] {$Z$};
\node (Q1') [left of=Z] {$Q_1'$};
\node (Q1) [left of=Q1'] {$Q_1$};
\node (Q2) [right of=Z] {$Q_2'$};
\node (Q2') [right of=Q2] {$Q_2$};
\node (1) [below of=Z] {$1$};
\end{scope}
\draw[arrow]
    (D8) edge (V)
        edge (V')
        edge (C4)
    (V) edge (Q1)
        edge (Q1')
        edge (Z)
    (V') edge (Q2)
        edge (Q2')
        edge (Z)
    (C4) edge (Z)
    (1) edge (Q1)
        edge (Q1')
        edge (Q2)
        edge (Q2')
        edge (Z)
        ;
\path
    (Q1) -- node{$\sim$} (Q1')
    (Q2) -- node{$\sim$} (Q2')
    ;
\draw[dashed,auto,swap]
    (Q1') -- node{$\cF$} (Z)
    (Z) -- node{$\cF$} (Q2)
    ;
\end{tikzpicture}
\]
Each sign $\sim$ in the diagram indicates that the two subgroups are conjugate in $D_8$, and each \begin{tikzpicture}[baseline=-.1cm]\draw[dashed,auto,swap] (0,0) -- node{$\cF$} (.5,0);\end{tikzpicture} indicates that the subgroups are conjugate in $\cF$ but not in $D_8$. Finally, the circles indicate the $\cF$-centric subgroups of $D_8$.

The fusion system $\cF$ is generated by an outer automorphism $\alpha\colon V\to V$ sending $Q_1$ to $Z$ and an outer automorphism $\beta\colon V'\to V'$ sending $Q_2$ to $Z$. Let $\cL$ be the centric linking system for $\cF$. The $\cL$-automorphisms of $D_8$ are the elements of $D_8$ itself, and these form a single $(D_8,D_8)$-orbit of type $[D_8,id]$. All $\cL$-automorphisms of $C_4$ extend to $D_8$, hence they do not contribute to the biset $_S \fI_S$. Of the $24$ $\cL$-automorphisms of $V$ only $8$ of them extend to $D_8$; the remaining $16$ form a single $(D_8,D_8)$-orbit of type $[V,\alpha]$. Similarly the nonextendable $\cL$-automorphisms of $V'$ produce a biset orbit $[V',\beta]$.

The entire biset $_S \fI_S$ of nonextendable $\cL$-isomorphisms is thus isomorphic to 
\[_S \fI_S \cong [D_8,id] + [V,\alpha] + [V',\beta].\]
This however is not all of the characteristic biset for $\cF$. $\Lambda_\cF$ receives two additional orbits from the non-$\cF$-centric subgroups:
\[\Lambda_\cF = [D_8,id] + [V,\alpha] + [V',\beta] + [Q_1,\beta^{-1}\alpha] + [Q_2,\alpha^{-1}\beta].\]
Note that that $\beta^{-1}\alpha:Q_1\to Q_2$ is nonextendable, as is its inverse $\alpha^{-1}\beta:Q_2\to Q_1$, so each must be represented as a point-stabilizer in $\Lambda_\cF$.
\end{example}

\begin{definition}
An \emph{$\cF$-centric semicharacteristic biset} is an $\cF$-generated $(S,S)$-biset $\Omega$ with all point-stabilizers of the form $(P,\varphi)$ with $P$ an $\cF$-centric subgroup, and such that for all $\cF$-centric subgroups $P$ and $\varphi\in\cF(P,S)$, $\abs[\big]{\Omega^{(P,\varphi)}}=\abs[\big]{\Omega^{(P,\iota_P^S)}}=\abs[\big]{\Omega^{(\varphi P,\varphi^{-1})}}$.  If we also have $\abs{\Omega}/\abs S\not\equiv 0$ mod $p$, we say that $\Omega$ is a \emph{$\cF$-centric characteristic biset}.
\end{definition}

\begin{remark}\label{rmk:truncating general char gives centric char}
Each $\cF$-centric semicharacteristic biset $\Omega$ is by assumption $\cF$-stable on all the $\cF$-centric subgroups of $S$. By adding additional orbits $[Q,\psi]$ with $Q$ non-centric, as in the construction of Theorem \ref{F-basis}, we can construct a $\cF$-semicharacteristic biset from $\Omega$. Conversely, any semicharacteristic biset for $\cF$ can be truncated, by removing all orbits $[Q,\psi]$ with $Q$ non-centric, to give a $\cF$-centric semicharacteristic biset.

This provides a $1$-to-$1$ correspondence between the centric (semi)characteristic bisets for $\cF$ and those (semi)characteristic bisets of the form $\sum\limits_{\text{$\cF$-centric $(P)_\cF$}} c_P\cdot \Omega_P$ with $c_P\in \ZZ_{\geq 0}$.
\end{remark}

\begin{theorem}\label{thm:LinkingSystemsAreCentricPartOfMinimalCharacteristicBiset}
$_S\fI_S$ is an $\cF$-centric characteristic biset.  Moreover, it is the unique minimal $\cF$-centric characteristic biset for $\cF$, and thus is the $\cF$-centric part of the minimal characteristic biset for $\cF$.
\end{theorem}

\begin{proof}
Suppose that $(R,\chi)$ is the stabilizer of $P\xra\fg Q\in\fI$, so that $\chi(r)\cdot\fg\cdot r^{-1}=\fg$ for all $r\in R$.  The definition of the $(S,S)$-action forces  $R\leq N_S(P)$ and $\chi (R)\leq N_S(Q)$.  $\fg$ is nonextendable, so Corollary \ref{conditions for nonextendability in L} implies  $R\leq P$ and $\chi=c_\fg\big|_R$.  As $(P,c_\fg)$ fixes $\fg$, it follows that $(R,\chi)=(P,c_\fg)$, so  every point-stabilizer of $_S\fI_S$ is a $\cF$-twisted diagonal subgroup whose source  is $\cF$-centric.

We now demonstrate $\cF$-stability on the $\cF$-centrics.  Let $P$ be an $\cF$-centric subgroup and $(P,\varphi)$ an $\cF$-twisted diagonal subgroup; we claim $\abs*{({}_S\fI_S)^{(P,\varphi)}}=\abs{Z(P)}$.  If $A\xra\fh B\in({}_S\fI_S)^{(P,\varphi)}$, then $\ph(p)\cdot \fh\cdot p^{-1}=\fh$ for all $p\in P$, so the above argument  gives  $P\leq A$ and $\varphi=c_\fh\big|_P$.  In other words, there is a natural bijection between the fixed points of $(P,\varphi)$ and the elements of $\fI$ that restrict to $\varphi$.  As every morphism of $\cL$ is epi and mono, an element of $\fI$ is uniquely determined by its restriction and conversely, so the number of $(P,\varphi)$-fixed points is the number of isomorphisms in $\cL$ with source $P$ that project to $\varphi$ in $\cF$.  By the linking system axioms there are  $\abs{Z(P)}$ such isomorphisms, proving the claim.

Finally, we show that $_S\fI_S$ is minimal.  If $\Stab_{S\x S}(\fg)=(P,\iota_P^S)$, we must have $c_\fg=\id_P$, which is only nonextendable when $P=S$.  Thus if $(P,\iota_P^S)$ is a stabilizer, we must have $P=S$.  Finally, as $\abs*{[S,\id]^{(S,\id)}}=\abs{Z(S)}=\abs*{({}_S\fI_S)^{(S,\id)}}$, we conclude that there is exactly one orbit with stabilizer $(S,\id)$, and we are done.
\end{proof}

\section{The linking-system-free centric minimal characteristic biset}\label{sec:LinkingSystemFreeMinimalBisets}

In this section we determine the minimal $\cF$-centric characteristic biset for a saturated fusion system $\cF$ in purely fusion-theoretic terms without assuming the existence of a linking system for $\cF$. The key for the argument is Puig's result, here recorded as Proposition \ref{every morphism in the centric orbit category is epi reason} and Corollary \ref{extensions in centric fusion system are unique up to conjugation by central elements}, describing the degree to which a morphism between $\cF$-centric subgroups has unique extensions.

\begin{remark}
Fix  $\varphi\in\cF_\iso(P,Q)$.  For $a,b\in S$, set $\psi:=c_a\circ\varphi\circ c_b\in\cF_\iso(\lc{b^{-1}}P,\lc aQ)$.  If $\lc{b^{-1}}P\leq \lc{b^{-1}}{\widetilde P}$ and $\widetilde\psi\in\cF(\lc{b^{-1}}{\widetilde P,S)}$ extends $\psi$, then $c_a^{-1}\circ\widetilde\psi\circ c_b^{-1}\in\cF(\widetilde P,S)$ extends $\varphi$.  Thus $\varphi$ is nonextendable if and only of $c_a\circ\varphi\circ c_b$ is nonextendable for all $a,b\in S$.
\end{remark}

\begin{notation}
Let $\sI$ be a set of representatives of the equivalence classes of nonextendable $\cF$-isomorphisms between $\cF$-centric subgroups of $S$, where $\varphi\sim\varphi'$ if there exist $a,b\in S$ such that $\varphi'=c_a\circ\varphi\circ c_b$.
\end{notation}


\begin{prop}{\cite[Proposition 3.3]{PuigFrobeniusCategories}}\label{every morphism in the centric orbit category is epi reason}
Let $P\leq Q\leq S$ be two $\cF$-centric subgroups.  If $\psi_1,\psi_2\colon Q\to S$ are such that $\psi_1|_P=\psi_2|_P$, then there is some $z\in Z(P)$ such that $\psi_2= \psi_1\circ c_z|_Q\in\cF(Q,S)$.
\end{prop}

\begin{remark}
In fact, Puig's formulation deals with $\cF$-quasicentric subgroups (``nilcentralized'' in his terminology), a more general class of subgroups than the $\cF$-centrics.  The original statements is:  If $P\leq Q\leq S$ are $\cF$-quasicentric subgroups with $\psi_1,\psi_2\in\cF(Q,S)$ such that $\psi_1|_P=\psi_2|_P=:\varphi\in\cF(P,S)$ and $\varphi P$ is fully $\cF$-centralized, then there is some $z\in C_S(\varphi P)$ such that $\psi_2=c_z\circ\psi_1\in\cF(Q,S)$.  In the case that $P$ is $\cF$-centric, we have $C_S(\varphi P)=Z(\varphi P)$.  Thus $z=\varphi(z')=\psi_1(z')$ for some $z'\in Z(P)$, and $c_z\circ\psi_1=\psi_1\circ c_{z'}$, and we recover the above formulation.
\end{remark}

\begin{cor}\label{extensions in centric fusion system are unique up to conjugation by central elements}
If $P\leq S$ is $\cF$-centric then each $\varphi\in\cF(P,S)$ has a unique nonextendable extension, up to precomposition with conjugation by elements of $Z(P)$.  In other words, if $\psi_1\in\cF(Q_1,S)$ and $\psi_2\in\cF(Q_2,S)$ are both nonextendable extensions of $\varphi$, then $Q_1=Q_2$ and there is some $z\in Z(P)$ such that $\psi_2=\psi_1\circ c_z$.
\end{cor}

\begin{proof}   We break the proof into three steps.

\noindent (1) \emph{Conjugate uniqueness on intersections}:  First suppose that we have two (possibly extendable) extensions $\eta_1\in\cF(A_1,S)$ and $\eta_2\in\cF(A_2,S)$ of $\varphi\in\cF(P,S)$, and set $B:=A_1\cap A_2.$ Then $\eta_1|_B$ and $\eta_2|_B$ are two extensions of $\varphi$ with the same source $B$, so by Proposition \ref{every morphism in the centric orbit category is epi reason} there is some $z\in Z(P)$ such that $\eta_2|_B=\eta_1\circ c_z|_B$.  Thus, up to precomposition with conjugation by a central element of $P$, we may assume that any two extensions of $\varphi$ agree wherever both are defined.

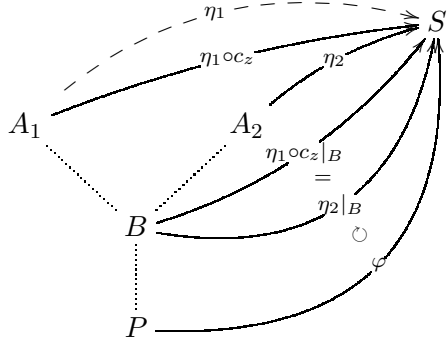
\[
\xymatrix{
&&&&S\\
A_1\ar@/^2pc/@{-->}[rrrru]|{\eta_1}\ar@/^/[rrrru]|{\eta_1\circ c_z}&&A_2\ar@/^/[rru]|{\eta_2}\\
&B\ar@{.}[ul]\ar@{.}[ur]\ar@/_1pc/[rrruu]|{\eta_1\circ c_z|_B}="a"\ar@/_3pc/[rrruu]|{\eta_2|_B}="b"
\ar@{}"a";"b"|=\\
&P\ar@/_4pc/[rrruuu]|\varphi="c"\ar@{.}[u]
\ar@{}"c";"b"|\circlearrowright
}
\]

\noindent (2) \emph{Existence and conjugate uniqueness of normal extensions}:  Suppose now we have two extensions (still possibly extendable) $\eta_1\in\cF(A_1,S)$ and $\eta_2\in\cF(A_2,S)$  of $\varphi\in\cF(P,S)$, and that $P\trianglelefteq A_i$, $i=1,2$.  Set $C:=\langle A_1,A_2\rangle\leq N_S(P)$.  Recall that $N_\varphi$, the extender of $\varphi$, is the largest subgroup of $N_S(P)$ for which there exists an extension of $\varphi$ (because all subgroups in sight are $\cF$-centric).  By assumption, we have $A_i\leq N_\varphi$, $i=1,2$. Hence $C\leq N_\varphi$ as well, and there is some $\zeta\in\cF(C,S)$ that extends $\varphi$.  As $\zeta|_{A_i}$ and $\eta_i$ are two morphisms  in $\cF(A_i,S)$ that extend $\varphi$, Proposition \ref{every morphism in the centric orbit category is epi reason} implies that there is some $z_i\in Z(P)$ such that $\eta_i\circ c_{z_i}=\zeta|_{A_i}\in\cF(A_i,S)$.  Thus, up to composition with conjugation by a central element of $P$, the extensions $\eta_1$ and $\eta_2$ of $\varphi$ have a common extension, at least when $P$ is normal in the sources of the $\eta_i$.
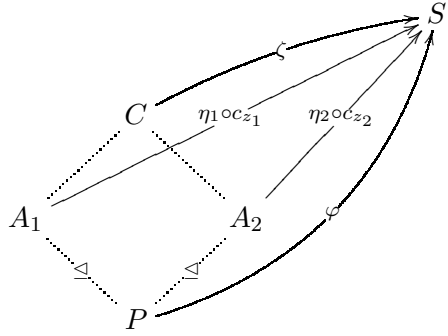
\[
\xymatrix{
&&&&S\\
&C\ar@/^/[rrru]|\zeta\\
A_1\ar@{.}[ur]\ar[urrrru]|{\eta_1\circ c_{z_1}}&&A_2\ar@{.}[ul]\ar[urru]|{\eta_2\circ c_{z_2}}\\
&P\ar@{.}[ur]|\trianglelefteq\ar@{.}[ul]|\trianglelefteq\ar@/_2pc/[ururur]|\varphi
}
\]

\noindent (3) \emph{General uniqueness}:  Finally, suppose that $\psi\in\cF(Q,S)$ is a \emph{nonextendable} extension of $\varphi$, and $\chi\in\cF(R,S)$ is some extension.  We will show that $R\leq Q$ and that there is some $z\in Z(P)$ such that  $\psi|_{R}=\chi\circ c_z\in\cF(R,S)$.  Clearly this will imply the overall result.

Set $B:=Q\cap R$.  By step (1), we may assume that $\psi|_B=\chi|_B$.  If $B=Q$, the nonextendability of $\psi$ forces $R=Q$, and we have our result.

Let us therefore induct on the index $[Q:B]$.  If $B\lneq Q$, then either $R\leq Q$ (and we're done) or $B$ is properly contained in both $N_Q(B)$ and $N_R(B)$.  Set $C:=\langle N_Q(B),N_R(B)\rangle$; by the second step, there is some $\eta\in\cF(C,S)$ that also extends $\varphi$, and such that $\eta|_{N_R(B)}=\chi\circ c_z|_{N_R(B)}$ for some $z\in Z(P)$.  As $[Q:C\cap Q]\leq[Q:N_Q(B)]<[Q:B]$, our inductive hypothesis gives us that $C\leq Q$.  In particular, $N_R(B)\leq Q$.  If $B=R\cap Q$ is properly contained in $R$ this yields a contradiction, so we conclude $R\leq Q$, and we're done.
\end{proof}

\begin{theorem}\label{minimal centric biset structure}
The $(S,S)$-biset $\Omega=\coprod\limits_{(\psi\colon Q\to Q')\in\sI}[Q,\psi]$ is the minimal $\cF$-centric characteristic biset.
\end{theorem}

\begin{proof}
Clearly $\Omega$ is $\cF$-generated, all point-stabilizers are $\cF$-twisted diagonals with source $\cF$-centric subgroups, $\Omega$ has precisely one orbit isomorphic to $[S,\id_S]$, and no other orbits are isomorphic to $[P,\iota_P^S]$.  Moreover, the only orbits of order $\abs{S}$ are those of the form $[S,\alpha]$ for $\alpha\in\Out_\cF(S)$.  Therefore $\abs{\Omega}/\abs{S}\equiv\abs{\Out_\cF(S)}\not\equiv 0$ modulo $p$.   Thus the only thing to do is show that $\Omega$ is $\cF$-stable on $\cF$-centric subgroups.  If $P\leq S$ is $\cF$-centric, $\varphi\in\cF(P,S)$, and $\omega\in\Omega$ has point-stabilizer $(Q,\psi)$, it is clear that $\omega\in\Omega^{(P,\varphi)}$ if and only if $(P,\varphi)\leq(Q,\psi)$, i.e., $P\leq Q$ and $\psi$ is an extension of $\varphi$.

Proposition \ref{extensions in centric fusion system are unique up to conjugation by central elements} implies that any two elements of $\Omega^{(P,\varphi)}$ must lie in the same $(S,S)$-orbit of $\Omega$:  If $\omega_i\in\Omega^{(P,\varphi)}$ have stabilizers $(Q_i,\psi_i)$, $i=1,2$, then $Q_1=Q_2$ and there is some $z\in Z(P)$ such that $\psi_2=\psi_1\circ c_z$.  As $(Q_1,\psi_1)$ and $(Q_1,\psi_1\circ c_z)$ are $(S\times S)$-conjugate, and  $\Omega$ has no two orbits that are isomorphic, we conclude that $\omega_1$ and $\omega_2$ lie in the same orbit.  Thus $\abs*{\Omega^{(P,\varphi)}}=\abs*{[Q,\psi]^{(P,\varphi)}}$, with $\psi$ our chosen representative in $\sI$ of the unique nonextendable extension of $\varphi$. By Proposition \ref{omnibus transport extender} (c),
\[
\abs*{[Q,\psi]^{(P,\varphi)}}=\frac{\abs{N_{\varphi,\psi}}\cdot \abs{C_S(\varphi P)}}{\abs{Q}}=\frac{\abs{N_{\varphi,\psi}}}{\abs{Q}}\cdot\abs{Z(P)}.
\]
We claim that $N_{\varphi,\psi}=Q$, so that order of the fixed point set is $\abs{Z(P)}$.  As this order depends only on the source of $\varphi$, it will follow that $\Omega$ is $\cF$-stable on $\cF$-centrics.

Recall that $N_{\varphi,\psi}=\left\{x\in N_S(P,Q)\Mid \psi\circ c_x\circ\varphi^{-1}\in\Hom_S(\varphi P,\psi Q)\right\}$.  If $q\in Q$, we have
\[
\psi\circ c_q\circ\varphi^{-1}=c_{\psi(q)}\circ\psi\circ\varphi^{-1}=c_{\psi(q)}\circ\varphi\circ\varphi^{-1}=c_{\psi(q)}\in\Hom_S(\varphi P,\psi Q).
\]
Therefore $Q\subseteq N_{\varphi,\psi}$.

For the other direction, fix $x\in N_{\varphi,\psi}$.   There is some $y\in N_S(\varphi P,\psi Q)$ such that
\[\psi\circ c_x\circ\varphi^{-1}=c_y\in\Hom_S(\varphi P,\psi Q),
 \textrm{ hence } c_y^{-1}\circ\psi\circ c_x|_P=\varphi\in\cF(P,\varphi P).
 \]  Thus $c_y^{-1}\circ\psi\circ c_x\in\cF(\lc{x^{-1}}Q,S)$ and $\psi\in\cF(Q,S)$ are two extensions of $\varphi$; by Proposition \ref{extensions in centric fusion system are unique up to conjugation by central elements} we have $\lc{x^{-1}}Q\leq Q$, hence $\lc{x^{-1}}Q=Q$, and  there is  $z\in Z(P)$ such that  $c_y^{-1}\circ\psi\circ c_x=\psi\circ c_z$.  Rewriting this as $\psi\circ c_{z x^{-1}}\circ\psi^{-1}=c_{y^{-1}}\in\Aut_S(\psi Q)$.  We have $zx^{-1}\in N_S(Q)$, so  that $z x^{-1}\in N_{\psi,\psi}=N_\psi$.  As $\psi$ is nonextendable (with target an $\cF$-centric, and hence fully $\cF$-centralized, subgroup),  $N_\psi=Q$ by the extension axiom for saturated fusion systems.  Thus $zx^{-1}\in Q$.  As $z\in Z(P)\leq P\leq Q$, we conclude $x\in Q$, and the proof is complete.
\end{proof}

\section{$K$-normalizers}\label{sec:KNormalizers}
For any saturated fusion system $\cF$ on $S$ and any fully $\cF$-normalized subgroup $P\leq S$ we can consider the associated normalizer fusion system $N_\cF(P)$; similarly for fully $\cF$-centralized $P$ and the centralizer fusion system $C_\cF(P)$. We might wonder whether it is possible to construct a minimal characteristic biset for $N_\cF(P)$ if we are given a minimal characteristic biset for $\cF$. In this section we introduce  a normalizer  subbiset $N_\Omega(P)\subseteq\Omega$ for a subgroup $P$ (resp., centralizer subbiset $C_\Omega(P)$) and show that in many cases this will be a characteristic biset for $N_\cF(P)$ (resp. $C_\cF(P)$).


\begin{definition}\label{K-normalizer definitions}
For $P\leq S$ and a subgroup $K\leq \Aut(P)$, we define the following concepts:
\begin{itemize}
\item The \emph{$K$-normalizer} of $P$ in $S$ is the group $N_S^K(P)=\left\{n\in N_S(P)\Mid c_n|_P\in K\right\}$.
\item If $Q\leq S$ is isomorphic to $P$ via an abstract group isomorphism $\varphi\colon P\to Q$, set $\lc\varphi K=\left\{\varphi\circ\alpha\circ\varphi^{-1}\Mid\alpha\in K\right\}\leq\Aut(Q)$.
\item  $P$ is \emph{fully $K$-normalized in $\cF$} if for all $\varphi\in\cF(P,S)$ we have $\abs[\big]{N_S^K(P)}\geq\abs[\big]{N_S^{\lc \varphi K}(\varphi P)}$.
\item The \emph{$K$-normalizer fusion system} is  the fusion system $N_\cF^K(P)$ on $N_S^K(P)$ with morphisms given by
\begin{multline*}
\hspace{\leftmargin}\Hom_{\cN_\cF^K(P)}(A,B)=\{\varphi\in\cF(A,B)\Mid\exists\textrm{ }\widetilde\varphi\in\cF(PA,PB)\\
\textrm{ s.t. }\widetilde\varphi|_A=\varphi,\textrm{ }\widetilde\varphi P=P,\textrm{ and }\widetilde\varphi|_P\in K\}.
\end{multline*}
\end{itemize}
\end{definition}

\begin{prop}[{\cite[Propositions 2.12 \& 2.15]{PuigFrobeniusCategories}}]\label{K-normalizer basics}
Let $P\leq S$ and $K\leq \Aut(P)$. Then $P$ is fully $K$-normalized in $\cF$ if and only if $P$ is fully $\cF$-centralized and $\Aut_S(P)\cap K\in\Syl_p(\cF(P)\cap K)$.

Furthermore, if $P$ is fully $K$-normalized in $\cF$, then $N_\cF^K(P)$ is a saturated fusion system on $N_S^K(P)$.
\end{prop}

\begin{example}
We have the following special cases of $K$-normalizers:
\begin{itemize}
\item If $K=\{\id_P\}$ is the trivial subgroup of $\Aut(P)$, then $N_\cF^K(P)=C_\cF(P)$ is the \emph{centralizer fusion subsystem of $P$}, whose underlying $p$-group is $C_S(P)$.
\item If $K=\Aut(P)$ is the full automorphism group of $P$, then $N_\cF^K(P)=N_\cF(P)$ is the \emph{normalizer fusion subsystem of $P$}, whose underlying $p$-group is $N_S(P)$.
\end{itemize}
\end{example}

In the following, we let $\Omega$ be some fixed $\cF$-semicharacteristic $(S,S)$-biset.

\begin{definition}
For any  $P\leq S$ and $K\leq\Aut(P)$,  the \emph{$K$-normalizer of P in $\Omega$} is
\[
N_\Omega^K(P):=\left\{\omega\in\Omega\Mid \Stab_{S\times S}(\omega)=(Q,\psi),\textrm{ }P\leq Q,\textrm{ }\psi P=P,\textrm{ and }\psi|_P\in K\right\}\subseteq\Omega.
\]

If $K=\{\id_P\}$, we denote the resulting \emph{centralizer of $P$ in $\Omega$} by $C_\Omega(P)$; if $K=\Aut(P)$, the \emph{normalizer of $P$ in $\Omega$} will be written $N_\Omega(P)$.
\end{definition}

\begin{remark}\label{effect of translation on twisted graph stabilizers}
If $\omega\in\Omega$ has stabilizer $(Q,\psi)$, then for any $a,b\in S$, we have $$\Stab_{S\x S}(a\cdot\omega\cdot b)=(Q^b,c_a\circ\psi\circ c_b).$$  In particular, $N_\Omega^K(P)$ need not be an $(S,S)$-biset.
\end{remark}

\begin{lemma}
$N_\Omega^K(P)$ is naturally an $(N_S^K(P),N_S^K(P))$-biset.
\end{lemma}

\begin{proof}
$n,m\in N_S^K(P)$, $\omega\in N_\Omega^K(P)$. If $\Stab_{S\x S}(\omega)=(Q,\psi)$, then $\Stab_{S\x S}(n\cdot\omega\cdot m)=(Q^m,c_n\circ\psi\circ c_m)$. It is clear that $P\leq Q^m$ and $(c_n\circ\psi\circ c_m)(P)=P$, and as each $c_m|_P$, $c_n|_P$, and $\psi|_P$ lie in $K$ it follows that $n\cdot \omega\cdot m\in N_\Omega^K(P)$, and the claim is proved.
\end{proof}

\begin{notation}
For the rest of this section, $P$ denotes some chosen subgroup of $S$, and we set $N:=N_S^K(P)$ and $\cN:=N_\cF^K(P)$.
\end{notation}

As the first step in deciding whether $N_\Omega^K(P)$ is $\cN$-characteristic, we describe the $(N,N)$-stabilizer of each element in $N_\Omega^K(P)$.

\begin{lemma}\label{stabilizers of points in automizers}
$\omega\in N_\Omega^K(P)$.  If $\Stab_{S\x S}(\omega)=(Q,\psi)$, then $\Stab_{N\x N}(\omega)=\left(N\cap Q,\psi\big|_{N\cap Q}\right)$.
\end{lemma}

\begin{proof}
The only nontrivial part  is that $\left(N\cap Q,\psi\big|_{N\cap Q}\right)\leq N\times N$, i.e., that $\psi(N\cap Q)\leq N$.  If $n\in N\cap Q$, then $n\in N_S(P)$, so $\psi(n)\in N_S(\psi P)=N_S(P)$. In addition we have $c_{\psi(n)}|_P=(\psi\circ c_n\circ\psi^{-1})|_P\in K$, so $\psi(n)\in N$ follows.
\end{proof}


\begin{lemma}\label{the number of extensions in normalizer subsystems depend only on the source}
$A,B,C\leq N$, $\varphi\in\cN_\iso(A,B)$, and $\psi\in\cN_\iso(A,C)$.  The number of extensions of $\varphi$ to $\widetilde\varphi\in\cN_\iso(PA,PB)$ equals the number of extensions of $\psi$ in $\cN_\iso(PA,PC)$.

Dually, if $A,B,C\leq N$, $\varphi\in\cN_\iso(A,C)$, and $\psi\in\cN_\iso(B,C)$, then the number of extensions of $\varphi$ to $\widetilde\varphi\in\cN_\iso(PA,PC)$ equals the number of extensions of $\psi$ in $\cN_\iso(PB,PC)$.
\end{lemma}

\begin{proof}
Any extension of  $\varphi\in\cN_\iso(A,B)$ with source $PA$ (whose existence is guaranteed by the definition of $\cN$) has image $PB$.  
If $\widetilde\varphi_1,\widetilde\varphi_2\in\cN_\iso(PA,PB)$ are extensions of $\varphi$, then $\widetilde\varphi_2^{-1}\circ\widetilde\varphi_1\in\cN(PA)$ and $\widetilde\varphi_2^{-1}\circ\widetilde\varphi_1\big|_A=\varphi^{-1}\circ\varphi=\id_A$. Let $G\leq \cN(PA)$ be the group of $\cN$-automorphisms of $PA$ that restrict to the identity on $A$, so that $G$ acts transitively on the set of lifts of $\varphi$ by precomposition.  This action is free, so the number of extensions of $\varphi$ to an $\cN$-isomorphism  with source $PA$ is $\abs{G}$.  The same is true for any other $\cN$-isomorphism with source $A$, and the result is proved.

The dual statement is proved by replacing each isomorphism with its inverse.
\end{proof}

\begin{prop}\label{K-normalizers of semicharacteristics bisets are semicharacteristic}
If $\Omega$ is an $\cF$-semicharacteristic biset, then $N_\Omega^K(P)$ is an $\cN$-semi\-charac\-ter\-istic $(N,N)$-biset.
\end{prop}

\begin{proof}
$N_\Omega^K(P)$ is $\cN$-generated: This is immediate from the definition  and Lemma \ref{stabilizers of points in automizers}.

$N_\Omega^K(P)$ is $\cN$-stable:  If $A\leq N$ and $\varphi\in\cN(A,N)$, we want to show that
\[
\abs*{(N_\Omega^K(P))^{(A,\varphi)}}=\abs*{(N_\Omega^K(P))^{(A,\iota_A^N)}}=\abs*{(N_\Omega^K(P))^{(\varphi A,\varphi^{-1})}}.
\]
If $P\leq A$, we claim that $(N_\Omega^K(P))^{(A,\varphi)}=\Omega^{(A,\varphi)}$.  The containment $\subseteq$ is obvious.  Suppose that $\omega\in\Omega^{(A,\varphi)}$, $\Stab_{S\x S}(\omega)=(Q,\psi)$.   We must have $(A,\varphi)\leq(Q,\psi)$, so that $A\leq Q$ and $\psi$ is an extension of $\varphi$.  It follows that $\psi P=\varphi P=P$ and $\psi|_P=\varphi|_P\in K$, so $\omega\in N_\Omega^K(P)$, as claimed.  $\Omega$ is $\cF$-stable, so  $N_\Omega^K(P)$ is $\cF$-stable on those twisted diagonal subgroups $(A,\varphi)$ such that $P\leq A$.

In general, given $A\leq N$ and an isomorphism $\varphi\in\cN_\iso(A,B)$, we consider the set $\left\{\widetilde\varphi_i\in\cN_\iso(PA,PB)\right\}_{i=1}^{n}$ of extensions of $\varphi$ to an isomorphism with source $PA$ (which must necessarily have target $PB$).  We then claim
\[
(N_\Omega^K(P))^{(A,\varphi)}=\coprod_{i=1}^n \Omega^{(PA,\widetilde\varphi_i)}.
\]

The union is disjoint:  If there are $1\leq i, j\leq n$ such that $\omega\in\Omega^{(PA,\widetilde\varphi_i)}\cap\Omega^{(PA,\widetilde\varphi_j)}$, then for all $x\in PA$, $
\tilde\ph_i(x)\cdot\omega\cdot x^{-1}=\omega=\tilde\ph_j(x)\cdot\omega\cdot x^{-1}$.  The left $S$-action on $\Omega$ is free, so $\widetilde\varphi_i(x)=\widetilde\varphi_j(x)$ for all $x\in PA$, hence $i=j$.

The equality holds:  For $\omega\in\Omega^{(PA,\widetilde\varphi_i)}$,  we have $\omega\in N_\Omega^K(P)$ because $\widetilde\varphi_iP=P$ and $\widetilde\varphi_i|_P\in K$. $(A,\varphi)\leq(PA,\widetilde\varphi_i)$, implies $\omega\in (N_\Omega^K(P))^{(A,\varphi)}$. Conversely, if $\omega\in (N_\Omega^K(P))^{(A,\varphi)}$, $\Stab_{S\x S}(\omega)=(Q,\psi)\leq S\times S$, then by definition of $N_\Omega^K(P)$ we have $P\leq Q$, $\psi P=P$, and $\psi|_P\in K$.  Thus $\psi(PA)\leq N$ and $\psi\big|_{PA}\in\cN_\iso(PA,PB)$ is an extension of $\varphi$.  Therefore there is some $\widetilde\varphi_i$ such that $\omega\in\Omega^{(PA,\widetilde\varphi_i)}$, proving the reverse containment.

Putting these claims together:
\begin{multline*}
\abs*{(N_\Omega^K(P))^{(A,\varphi)}}=\abs*{\coprod_{i=1}^n\Omega^{(PA,\widetilde\varphi_i)}}=\sum_{i=1}^n\abs*{\Omega^{(PA,\widetilde\varphi_i)}}
=n\cdot\abs*{\Omega^{(PA,\iota_{PA}^S)}}
\\=n\cdot\abs*{(N_\Omega^K(P))^{(PA,\iota_{PA}^N)}}.
\end{multline*}
The third equality uses the $\cF$-stability of $\Omega$; the fourth our  observation that $N_\Omega^K(P)$ is $\cN$-stable on those subgroups that contain $P$.  Note in particular that we have described $\abs*{(N_\Omega^K(P))^{(A,\varphi)}}$ as depending solely on the number of extensions $n$ of $\varphi$ to an isomorphism in $\cN$ with source $PA$.  By Lemma \ref{the number of extensions in normalizer subsystems depend only on the source}, this number depends not on $\varphi\in\cN_\iso(A,B)$, but only on the source $A$.  It follows that $\abs*{(N_\Omega^K(P))^{(A,\varphi)}}=\abs*{(N_\Omega^K(P))^{(A,\iota_A^N)}}$.  The dual result of Lemma \ref{the number of extensions in normalizer subsystems depend only on the source} implies that the number $n$ also can be seen to depend only on the target of the isomorphism; as $\id_A$ and $\varphi^{-1}\in\cN_\iso(\varphi A,A)$ have the same target, it follows that $\abs*{(N_\Omega^K(P))^{(A,\iota_A^N)}}=\abs*{(N_\Omega^K(P))^{(\varphi A,\varphi^{-1})}}$, and the $\cN$-stability of $N_\Omega^K(P)$ is proved.
\end{proof}

Aside:  The method of the proof of Proposition \ref{K-normalizers of semicharacteristics bisets are semicharacteristic} can be used to prove the following useful structure theorem for the minimal $\cF$-characteristic biset $\Lambda_\cF$:

\begin{prop}\label{prop:OpF contained in stabilizers}
If $(Q,\psi)$ is a point-stabilizer of $\Lambda_\cF$, then $O_p(\cF)\leq Q$.
\end{prop}

\begin{proof}
Let $\Xi$ be the $(S,S)$-biset obtained by applying the $(\cF\times\cF)$-stabilization process of Theorem \ref{F-basis} to $[S,\id_S]$  for the subgroups containing $O_p(\cF)$.  We will  show that $\Xi$ is $\cF$-stable, hence $\Xi=\Lambda_\cF$ and the result will follow.

So we must show for $A\leq S$ and $\varphi\in\cF_\iso(A,B)$ the following equalities:
\[
\abs*{\Xi^{(A,\varphi)}}=\abs*{\Xi^{(A,\iota_A^S)}}=\abs*{\Xi^{(\varphi A,\varphi^{-1})}}.
\]
If $O_p(\cF)\leq A$, these equalities hold by construction of $\Xi$.  Otherwise let $\widetilde\varphi_i$, $i=1,\ldots,n$, be the distinct extensions of $\varphi$ to elements of $\cF_\iso(O_p(\cF)\cdot A,O_p(\cF)\cdot B)$.  As in the proof of Proposition \ref{K-normalizers of semicharacteristics bisets are semicharacteristic} we can write $\Xi^{(A,\varphi)}=\coprod_{i=1}^n\Xi^{(PA,\widetilde\varphi_i)}$, so
\[
\abs*{\Xi^{(A,\varphi)}}=\sum_{i=1}^n\abs*{\Xi^{(PA,\widetilde\varphi_i)}}=n\cdot\abs*{\Xi^{(PA,\iota_{PA}^S)}},
\]
which depends only on the source $A$ by Lemma \ref{the number of extensions in normalizer subsystems depend only on the source}.  Therefore $\abs*{\Xi^{(A,\varphi)}}=\abs*{\Xi^{(A,\iota_A^S)}}$.  Dually we can show that the fixed-point order depends only on the target of the isomorphism in question, so $\abs*{\Xi^{(A,\iota_A^S)}}=\abs*{\Xi^{(\varphi A,\varphi^{-1})}}$.  This proves the result.
\end{proof}

Back on track:  We haven't made use of the saturation of $\cF$ yet in this section; now we will need to in order to guarantee the existence of characteristic bisets for $\cF$, in particular the unique minimal $\cF$-characteristic biset $\Lambda_\cF$ for $\cF$.

\begin{prop}\label{good K-normalizer subbisets of minimal characteristic bisets contain exactly one minimal characteristic biset}
Let $\Omega:=\Lambda_\cF$ be the minimal characteristic biset for $\cF$, and let $P\leq S$ fully $K$-normalized in $\cF$ for $K\leq\Aut(P)$.  If $K\leq\Inn(P)$ or $K\geq\Inn(P)$, then $N_\Omega^K(P)$ contains precisely one $(N,N)$-orbit isomorphic to $[N,\id_N]$.
\end{prop}

\begin{proof} We consider two cases.

(1)  $K=\{\id_P\}$ or $\Inn(P)\leq K$.

Fix $\omega\in N_\Omega^K(P)$, $\Stab_{N\x N}(\omega)=(N,\id)$ and $\Stab_{S\times S}(\omega)=(Q,\psi)$,  so $N\leq Q$ and $\psi|_N=\id_N$.  If $K=\{\id_P\}$, the definition of $C_\Omega(P)=N_\Omega^{\{\id_P\}}(P)$ shows we must also have $P\leq Q$ and $\psi|_P=\id_P$.  If $\Inn(P)\leq K$, the definition of $N_S^K(P)$ implies that $P\cdot C_S(P)\leq N$.  In either case, $P\cdot C_S(P)\leq Q$ and $\psi|_{P\cdot C_S(P)}=\id_{P\cdot C_S(P)}$.

As $P$ is fully $K$-normalized in $\cF$, it is fully $\cF$-centralized by Proposition \ref{K-normalizer basics}, so \cite[Proposition A.7]{BLO2} implies that  $P\cdot C_S(P)$ is $\cF$-centric.  As $\id_Q$ and $\psi$ both restrict to the same automorphism of $P\cdot C_S(P)$, Proposition \ref{every morphism in the centric orbit category is epi reason} says that there is some $z\in Z(P\cdot C_S(P))$ such that $\psi=\id_Q \circ c_z|_Q=c_z|_Q$.  Since $z\in S$, the $(S,S)$-bisets $[Q,c_z]$ and $[Q,\iota_Q^S]$ are isomorphic.  $Q$ is $\cF$-centric and $\Omega$ is minimal, so Theorem \ref{minimal centric biset structure} forces $Q=S$.

Thus all $\omega\in N_\Omega^K(P)$ with $\Stab_{N\x N}(\omega)=(N,\id)$ live in the same $(S,S)$-orbit $[S,\id_S]$, otherwise known as $S$ with its natural $(S,S)$-biset structure.  The subset of $_SS_S$ that lies in $N_\Omega(P)$ is  $N_S^K(P)=N$, so all such points of $N_\Omega^K(P)$ lie in the same $(N,N)$-orbit.

(2) $K\leq\Inn(P)$.

Before dealing with the nonidentity subgroups of $\Inn(P)$, we take a small detour to compare two different $K$-normalizers and their relation:  Let $K\leq\Aut(P)$ be arbitrary with $P$  fully $K$-normalized in $\cF$, and set  $L:=K\cdot\Inn(P)$.  Note that $\Inn(P)\trianglelefteq\Aut(P)$, so $L$ is in fact the product of $K$ and $\Inn(P)$, not merely the subgroup generated by the two.  We have $N_S^L(P)=P\cdot N_S^K(P)$.

Now, consider the natural inclusion $\iota\colon N_\Omega^K(P)\subseteq N_\Omega^L(P)$.  This is an $(N_S^K(P),N_S^K(P))$-equivariant map of bisets, hence $\iota$ induces a map on orbits \[\overline\iota\colon (N_S^K(P)\backslash N_\Omega^K(P)/N_S^K(P))\to(N_S^L(P)\backslash N_\Omega^L(P)/N_S^L(P)).\]
We claim that $\overline\iota$ is a bijection.

\emph{$\overline\iota$ is surjective}:  Suppose that $\omega\in N_\Omega^L(P)$,  $\Stab_{S\x S}(\omega)=(Q,\psi)$.  We have $P\leq Q$, $\psi P=P$, and $\psi|_P=\kappa\circ c_a$ for $\kappa\in K$ and $a\in P$.  By Remark \ref{effect of translation on twisted graph stabilizers}, the point $\omega\cdot a^{-1}$ has stabilizer $(\lc a Q,\psi\circ c_a^{-1})$ with $P\leq\lc a Q$, $(\psi\circ c_a^{-1}) P=P$, and $(\psi\circ c_a^{-1})|_P=\kappa\circ c_a\circ c_a^{-1}=\kappa\in K$.  Thus $\omega\cdot a^{-1}\in N_\Omega^K(P)$, and as $a\in P\leq N_S^L(P)$, we see $\overline\iota$ is surjective on orbits.

\emph{$\overline\iota$ is injective}:  Suppose that $\omega_1,\omega_2\in N_\Omega^K(P)$ have $(S,S)$-stabilizers $(Q_i,\psi_i)$, $i=1,2$.  We again have $P\leq Q_i$, $\psi_i P=P$, and $\psi_i|_P\in K$.  If $\omega_1$ and $\omega_2$ lie in the same $(N_S^L(P),N_S^L(P))$-orbit, there are elements $a,b\in N_S^L(P)$ such that $\omega_2=a\cdot\omega_1\cdot b$.  Since $N_S^L(P)=P\cdot N_S^K(P)$ we may write $b=p\cdot n$ for $n\in N_S^K(P)$ and $p\in P$.  As $P\leq Q_1$, we can write $\omega_2=a\cdot\omega_1\cdot p\cdot n=(a\psi_1(p))\cdot \omega_1\cdot n$.  By Remark \ref{effect of translation on twisted graph stabilizers}, the $(S,S)$-stabilizer of $(a\psi_1(p))\cdot \omega_1\cdot n$ is $\bigl((Q_1)^n,c_{a\psi_1(p)} \circ\psi_1\circ c_n\bigr)$.   We already have $(\psi_1\circ c_n)|_P\in K$, and the entire composite must restrict to an automorphism of $P$ that lies in $K$ because $\omega_2\in N_\Omega^K(P)$.  This forces $c_{a\psi_1(p)}|_P\in K$, or $a\psi_1(p)\in N_S^K(P)$.  Thus $\omega_1$ and $\omega_2$ live in the same $(N_S^K(P),N_S^K(P))$-orbit, and injectivity is proved.

In fact, we have shown more:  Given any subgroup $H\leq\Aut(P)$ of automorphisms such that $H\leq K\leq L:=H\cdot\Inn(P)=K\cdot\Inn(P)$, we have that the inclusions $N_\Omega^H(P)\subseteq N_\Omega^L(P)$ and $N_\Omega^K(P)\subseteq N_\Omega^L(P)$ both induce bijections on orbits, so in fact the third natural inclusion $N_\Omega^H(P)\subseteq N_\Omega^K(P)$ must induce a bijection on orbits as well.

In particular, consider the case that $H=\{\id_P\}$, $L=\Inn(P)$, and $K\leq\Inn(P)$ is arbitrary.  Then $N_S^H(P)=C_S(P)$, and we've already seen that there is a unique $(C_S(P),C_S(P))$-orbit of $C_\Omega(P)$ with stabilizer $(C_S(P),\id_{C_S(P)})$.  There is some $\omega\in\Omega$ that has $(S,S)$-stabilizer $(S,\id)$, so $\omega\in C_\Omega(P)\subseteq N_\Omega^K(P)$ and has $(N,N)$-stabilizer $(N,\id_N)$ as an element of $N_\Omega^K(P)$.  Suppose that there is some other $\omega'\in N_\Omega^K(P)$ with $(N,N)$-stabilizer $(N,\id_N)$.  Then $\omega'\in C_\Omega(P)$ and has $(C_S(P),C_S(P))$-stabilizer $(C_S(P),\id_{C_S(P)})$, and as we have already proved our result for the centralizer biset, we conclude that $\omega$ and $\omega'$ must lie in the same $(C_S(P),C_S(P))$-orbit, and hence in the same $(N,N)$-orbit as well.  This proves the result for arbitrary subgroups of $\Inn(P)$.
\end{proof}

In the course of the proof of Proposition \ref{good K-normalizer subbisets of minimal characteristic bisets contain exactly one minimal characteristic biset} we made use of the following interesting fact, which we record here for ease of reference:

\begin{prop}\label{orbits of K-normalizers are unchanged by Inn(P)}
Let $\Omega$ be a semicharacteristic biset for $\cF$, $P$ a subgroup of $S$, and $H,K\leq\Aut(P)$ two groups of automorphisms satisfying
\[H\leq K\leq H\cdot\Inn(P)=K\cdot\Inn(P).\]
Then the natural inclusion $N_\Omega^H(P)\subseteq N_\Omega^K(P)$ induces a bijection on orbits:
\[(N_S^H(P)\backslash N_\Omega^H(P)/N_S^H(P))\cong(N_S^K(P)\backslash N_\Omega^K(P)/N_S^K(P)).\]
\end{prop}

\begin{remark}
We could use Propositions \ref{good K-normalizer subbisets of minimal characteristic bisets contain exactly one minimal characteristic biset} and \ref{orbits of K-normalizers are unchanged by Inn(P)} to reprove Puig's main theorem on $K$-normalizers (cf. \cite[Proposition 21.11]{PuigBook}):  If $\Omega$ is a characteristic biset for $\cF$ with $P\leq S$ and $K\leq\Aut(P)$ given so that $P$ is fully $K$-normalized in $\cF$, then $N_\Omega^K(P)$ is a characteristic biset for $N_\cF^K(P)$.

[Sketch of proof:  $N_\Omega^K(P)$ is always $\cN$-semicharacteristic by Proposition \ref{K-normalizers of semicharacteristics bisets are semicharacteristic}, so we only need show that $p\nmid\abs{N_\Omega^K(P)}/\abs{N_S^K(P)}$.  In the case that $K$ contains or is contained in $\Inn(P)$, this is a direct calculation based on Proposition \ref{good K-normalizer subbisets of minimal characteristic bisets contain exactly one minimal characteristic biset}; in the general case one can use Proposition \ref{orbits of K-normalizers are unchanged by Inn(P)} to show $\abs{N_\Omega^K(P)}/\abs{N_S^K(P)}=\abs{N_\Omega^{L}(P)}/\abs{N_S^L(P)}$ where $L:=K\cdot\Inn(P)$, and that $P$'s being fully $K$-normalized in $\cF$ implies that it is also fully  $L$-normalized.  From this the result follows.]

In particular, the existence of a $N_\cF^K(P)$-characteristic biset implies that $N_\cF^K(P)$ is a saturated fusion system.  There is little gained by reproving this result in detail; instead we will assume it and derive the following more precise formulation.
\end{remark}

\begin{theorem}\label{K-normalizers of minimal are minimal}
Suppose that $\Omega=\Lambda_\cF$ is the minimal characteristic biset for $\cF$.  Suppose $P\leq S$ and $K\leq\Aut(S)$ such that $K$ either contains or is contained in $\Inn(P)$.

If $P$ is fully $K$-normalized in $\cF$, then $N_\Omega^K(P)$ is a characteristic $(N,N)$-biset for $\cN=N_\cF^K(P)$ that contains precisely one copy of $\Lambda_{\cN}$, the minimal characteristic biset for $\cN$.

Moreover, if $P$ is $\cF$-centric, then $N_\Omega^K(P)=\Lambda_\cN$.
\end{theorem}

\begin{proof}
By \cite[Proposition 21.11]{PuigBook}, $\cN$ is a saturated fusion system on $N$, hence our parameterization of semicharacteristic bisets applies.   $N_\Omega^K(P)$ is $\cN$-semicharacteristic by Proposition \ref{K-normalizers of semicharacteristics bisets are semicharacteristic}, and by Theorem \ref{parameterization of semicharacteristic bisets} the number of copies of $\Lambda_\cN$ contained in $N_\Omega^K(P)$ is equal to the number of orbits isomorphic to $[N,\id_N]$.  By Proposition \ref{good K-normalizer subbisets of minimal characteristic bisets contain exactly one minimal characteristic biset}, there is a unique such $(N,N)$-orbit, proving the first statement.

Now, suppose that $P$ is $\cF$-centric.  To show that $N_\Omega^K(P)=\Lambda_\cN$, it suffices to show that there are no other minimal $\cN$-semicharacteristic bisets beyond $\Lambda_\cN$ contained in $N_\Omega^K(P)$.  Suppose that $\omega\in N_\Omega^K(P)$ has $(N,N)$-stabilizer $(A,\iota_A^N)$, then $P\leq A\leq N$ by Proposition \ref{prop:OpF contained in stabilizers}.  Then the $(S,S)$-stabilizer of $\omega$ is $(Q,\psi)$, with $A=Q\cap N$ and $\psi|_A=\id_A$.  All groups in sight are $\cF$-centric by assumption that $P$ is, so we may use Theorem \ref{every morphism in the centric orbit category is epi reason} to conclude that $\psi=c_z|_Q$ for some $z\in Z(A)$.  Therefore $[Q,\psi]=[Q,c_z]=[Q,\iota_Q^S]$, and we know from Theorem \ref{minimal centric biset structure} that the only such orbit in $\Lambda_\cF$ when $Q$ is $\cF$-centric is $[S,\id_S]$.  We conclude that $Q=S$ and $N\leq Q$.  Therefore the only point-stabilizer of $N_\Omega^K(P)$ of the form $(A,\iota_A^N)$ is $(N,\id_N)$, so the only $\cN$-semicharacteristic bisets contained in $N_\Omega^K(P)$ are copies of $\Lambda_\cN$. As we have seen that there is exactly one of these, we have $N_\Omega^K(P)=\Lambda_\cN$, as claimed.
\end{proof}

\begin{conj}\label{P need not be F-centric for good K-normalizers of minimals to be minimal?}
$P$ need not be $\cF$-centric for the conclusions of Theorem \ref{K-normalizers of minimal are minimal} to hold:  If $\Omega=\Lambda_\cF$ is the minimal characteristic biset for $\cF$ and we are given $P\leq S$ and $K\leq\Aut(P)$ such that $K$ contains or is contained in $\Inn(P)$, and if $P$ is fully $K$-normalized in $\cF$, then $N_\Omega^K(P)=\Lambda_{\cN}$, the minimal characteristic biset for $\cN$.
\end{conj}

\begin{counter}\label{Q_8 semidirect Z/3 counterexample}
There can be no analogue of  Conjecture \ref{P need not be F-centric for good K-normalizers of minimals to be minimal?} that completely relaxes the conditions on $K\leq\Aut(P)$ in Proposition \ref{good K-normalizer subbisets of minimal characteristic bisets contain exactly one minimal characteristic biset} and Theorem \ref{K-normalizers of minimal are minimal} and still have the conclusions hold:

Let $\ZZ/3$ act on $Q_8$ by permuting the elements $i,j,k$ cyclically.  Set $G:=Q_3\rtimes \ZZ/3$, $\cF:=\cF_{Q_8}(G)$, and $K=\ZZ/3\leq\Aut(Q_8)$.  Note that $Q_8$ is fully $K$-normalized in $\cF$.  Since $K$ is a $2'$-group, we have $N_{Q_8}^K(Q_8)=Z(Q_8)\cong\ZZ/2$.  If $\kappa$ is a generator for $K$, one easily checks that the minimal $\cF$-characteristic biset is $\Lambda_\cF=[Q_8,\id_{Q_8}]\amalg[Q_8,\kappa]\amalg[Q_8,\kappa^2]$.  One can further calculate that $N_{\Lambda_{\cF}}^K(Q_8)=3\cdot[Z(Q_8),\id_{Z(Q_8)}]$, contrary to the conclusion of Proposition \ref{good K-normalizer subbisets of minimal characteristic bisets contain exactly one minimal characteristic biset}.
\end{counter}

\begin{remark}
Counterexample \ref{Q_8 semidirect Z/3 counterexample} shows in particular that there must be some condition imposed on $K\leq\Aut(P)$ in general to guarantee that $N_{\Lambda_\cF}^K(P)=\Lambda_{N_\cF^K(P)}$.  We have seen that it is enough (when $P$ is $\cF$-centric) to assume that $K$ either contains or is contained in $\Inn(P)$.  While it is possible that one could find a larger class of subgroups of $\Aut(P)$ for which the conclusion of Theorem \ref{K-normalizers of minimal are minimal} holds, we have at least already covered the most important examples with our current formulation: If $K=\{\id\}$ or $K=\Aut(P)$ we get the minimal characteristic bisets $C_\Omega(P)$ and $N_\Omega(P)$ for the fusion systems $C_\cF(P)$ and $N_\cF(P)$, respectively.  We also cover the cases of the subsystems $Q\cdot C_\cF(Q)$ (on $Q\cdot C_S(P)$) and $N_P(Q)\cdot C_\cF(Q)$ (on $N_S(P)$) corresponding to the cases $K=\Inn(P)$ and $K=\Aut_S(P)$, respectively  (cf. \cite[Definition 3.1]{LinckelmannIntro}).
\end{remark}

\bibliography{Sources}
\bibliographystyle{alphanum}

\end{document}